\documentclass[11pt,a4paper]{article}
\usepackage[margin=1in]{geometry}
\usepackage[T1]{fontenc}
\usepackage[round,authoryear]{natbib}
\usepackage{amsmath,amssymb,amsthm,mathtools}
\usepackage{booktabs}
\usepackage{graphicx}
\ifdefined\pdfminorversion\pdfminorversion=7\fi
\usepackage{float}
\usepackage{microtype}
\usepackage{needspace}
\usepackage{multirow}
\usepackage{xcolor}

\usepackage[hidelinks]{hyperref}
\hypersetup{pdftitle={Sieve-free ReLU Bayesian priors with near-minimax posterior contraction},
  pdfauthor={Hien Duy Nguyen, Jacob Westerhout, Julyan Arbel}}

\newtheorem{theorem}{Theorem}
\newtheorem{proposition}{Proposition}
\newtheorem{lemma}{Lemma}

\theoremstyle{remark}
\newtheorem{remark}{Remark}
\newcommand{\Xcal}{\mathcal X}
\newcommand{\Hcal}{\mathcal H}
\newcommand{\Dcal}{\mathcal D}
\newcommand{\Lcal}{\mathcal L}
\newcommand{\Ccal}{\mathcal C}
\newcommand{\PP}{\mathbb P}
\newcommand{\E}{\mathbb E}
\newcommand{\norm}[1]{\lVert#1\rVert}
\newcommand{\abs}[1]{\lvert#1\rvert}
\DeclareMathOperator{\Unif}{Unif}

\title{Sieve-free ReLU Bayesian priors with near-minimax posterior contraction}
\author{Hien Duy Nguyen$^{1,2}$ \quad Jacob Westerhout$^{3}$ \quad Julyan Arbel$^{4}$\\[0.8em]
\parbox{0.96\textwidth}{\centering\small
$^{1}$Department of Mathematical and Physical Sciences,\\
La Trobe University, Melbourne, Australia\\
$^{2}$Institute of Mathematics for Industry, Kyushu University, Fukuoka, Japan\\
$^{3}$School of Mathematics and Physics, The University of Queensland, Brisbane, Australia\\
$^{4}$Univ. Grenoble Alpes, Inria, CNRS, Grenoble INP, LJK, Grenoble, France}}
\date{}

\begin{document}
\maketitle
\raggedbottom

\begin{abstract}
We establish fractional-posterior contraction for shallow ReLU networks
in Gaussian random-design regression with a $\beta$-H\"older target,
using priors independent of the sample size and smoothness parameters.
For any fixed design law $P_X$ on $[0,1]^d$, we obtain contraction in
R\'enyi divergence and, after clipping posterior draws at a bound on
the target range, in squared $\mathcal{L}^2(P_X)$ loss at rate
$(\log n/n)^{2\beta/(2\beta+d)}$.
This matches the classical H\"older minimax rate up to a logarithmic
factor.
These priors have random finite width and
use width-dependent uniform or Gaussian weight distributions, or
width-independent Student $t$ distributions, including Cauchy distributions.
\end{abstract}

\section{Introduction}

A central question in the Bayesian theory of neural networks is whether
posteriors under network priors attain optimal frequentist rates of
estimation. For deep ReLU networks, near-minimax contraction has been
established under spike-and-slab priors whose nonzero network parameters
are uniform on a bounded interval \citep{polsonrockova2018}, under
Gaussian and other network-parameter laws with positive densities
\citep{kongkim2025}, and under heavy-tailed network-parameter laws such as Student
$t$ and Cauchy distributions \citep{castilloegels2025}. In each case, however,
the joint prior changes with the sample size $n$: some part of the
architecture, its distribution, or the network-parameter scale is chosen as a
function of $n$. This allows the prior
specification to be calibrated to the statistical resolution available
at sample size $n$, but means that a different prior is used as the
amount of data changes. We show that this dependence on $n$ is
unnecessary. Here, sieve-free refers to a prior fixed across sample
sizes; this use of the term should not be confused with mixtures over
finite-dimensional models, which are also called sieve priors
\citep{arbel2013}.

Our main aim is to isolate the mechanism that allows a single fixed
prior to attain the same rates. A fixed prior cannot be tuned to one
statistical accuracy; it must instead allocate enough probability
simultaneously across all relevant approximation scales. For
one-hidden-layer networks, width provides a natural index for these
scales. Given a target and an accuracy $\varepsilon>0$, we use the
approximation bounds of \citet{yangzhou2024,yangzhou2025} to select an approximating
network of suitable width $k$, and ask how much probability the same
prior assigns to a neighbourhood of its parameters. The hidden-layer coefficients
and biases of these approximators remain bounded, while the required
output coefficients grow at most polynomially with the approximation
complexity. These bounds can be dominated by rules depending only on
the sampled width and the input dimension. This suggests a hierarchical
construction: specify once and for all a law on all finite widths and,
conditional on width $k$, a law on the corresponding network parameters. Neither
$n$ nor the unknown smoothness or radius of the target need enter the
prior.

The argument separates the approximation problem from a local
prior-mass requirement: the prior must assign sufficient probability
to a suitable width and to accurate parameters at that width.
We verify this requirement for three
qualitatively different conditional network-parameter families: bounded uniform
laws with width-dependent output support, Gaussian laws with a
width-dependent output variance, and fixed Student $t$ laws, including
the Cauchy distribution. The same calculation applies to fixed network-parameter
densities satisfying a polynomial lower bound. These examples identify
local mass near approximating networks as the property that permits
adaptation.

We work with one-hidden-layer networks in Gaussian random-design
regression, with fixed input dimension $d\geq1$ and design law $P_X$
on $[0,1]^d$. Width indexes approximation complexity, and uniform
approximation gives likelihood neighbourhoods for arbitrary design laws.
Following \citet{bhattacharya2019} and \citet{castilloegels2025},
we study a fractional posterior, which raises the likelihood to a fixed
power $0<\eta<1$. Its contraction argument requires only sufficient
local prior mass, avoiding the additional global control of large models
needed for ordinary-posterior contraction. The order-$\eta$ R\'enyi
divergence arises directly from the fractional likelihood ratio.
For H\"older smoothness $0<\beta\leq1$ and radius $M>0$, the
fractional posterior under the same fixed prior contracts in R\'enyi
divergence at order $(\log n/n)^{2\beta/(2\beta+d)}$, uniformly over
the corresponding H\"older ball.

R\'enyi contraction isolates this mechanism, but squared
prediction error is a stronger and more familiar
frequentist loss. The passage between the two creates a separate
difficulty: posterior draws under our fixed priors need not be
uniformly bounded, and large discrepancies on sets of small
$P_X$-probability can have a much greater effect on squared error than
on R\'enyi divergence. Rather than impose additional restrictions on
the prior or a clipping level depending on the unknown target radius,
we clip posterior draws at a level estimated from the data. Under a
design density bounded away from zero, this level contains the range
of the target with exponentially high probability. Once both functions
are bounded, R\'enyi divergence controls squared prediction error.
The resulting posterior mean has risk
$(\log n/n)^{2\beta/(2\beta+d)}$ under the lower density bound.
If the density is also bounded above, the classical H\"older minimax
lower bound shows that this rate is optimal up to a logarithmic factor.
An alternative clipping construction retains the same upper
risk rate under arbitrary design laws, at the cost of an additional
polynomial term in the posterior-tail bound. Clipping also ensures
that the posterior mean exists. For Cauchy and Student network-parameter priors with
at most one degree of freedom, the unclipped mean can fail to exist
when a unit is inactive at every observed covariate; see
Appendix~\ref{app:unclipped-mean}. Whether squared prediction-loss
contraction holds without clipping, or the unclipped mean attains
the same risk rate under conditions ensuring its existence, remains open.

To the best of our knowledge, this is the first adaptive near-minimax
fractional-posterior contraction result for dense shallow ReLU networks
under a prior independent of the sample size, with bounded parameter
distributions conditional on width.

\subsection{Approximation and frequentist estimation}\label{sec:related}

Classical shallow-network rates such as those of \citet{barron1993}
concern Fourier-moment classes rather than H\"older balls.
For deep ReLU networks, quantitative approximation bounds
\citep{yarotsky2017} underpin near-minimax estimation for
compositional targets \citep{schmidthieber2020}, fully connected
architectures \citep{kohlerlanger2021}, and Besov and mixed-smoothness
classes \citep{suzuki2019}.

We use the shallow-network approximation rate $k^{-\beta/d}$ for width $k$ and
output-coefficient bounds of \citet[Cor.~2.4(3)]{yangzhou2025}, valid for
$0<\beta<(d+3)/2$; see also \citet{yangzhou2024}. Rather than select
an architecture for a frequentist procedure, these bounds identify
accurate parameter neighbourhoods whose probability we evaluate under
one fixed prior. We compare the resulting risk with the classical
H\"older minimax lower bound of \citet{stone1982}.

\subsection{Bayesian contraction and architecture selection}

Random widths already appear in the shallow-network consistency
results of \citet{lee2000}; hierarchical dimension priors are also
standard in Bayesian nonparametrics \citep{arbel2013}.
For deep ReLU networks, \citet{polsonrockova2018} obtain near-minimax
contraction with uniform-slab spike-and-slab priors and random width,
but retain other $n$-dependent architectural choices. Related
approaches use sparse connections
\citep{leelee2022,leelinparkjeong2025}, hidden-unit selection
\citep{kong2023}, or variational inference
\citep{cherief2020,bai2020,ohnlin2024}.

For dense ordinary posteriors, \citet{kongkim2025} cover Gaussian and
other network-parameter laws through a local density lower bound on every
bounded parameter box, with network dimensions varying with $n$.
Our condition instead controls neighbourhoods of
approximation-compatible parameters, allowing width-dependent
bounded supports.

\subsection{Heavy tails and fractional posteriors}

Heavy-tailed series priors adapt to unknown smoothness
\citep{agapioucastillo2024}. For deep ReLU networks,
\citet{castilloegels2025} obtain near-minimax fractional-posterior
contraction with heavy-tailed network parameters and $n$-dependent network
specifications and coefficient scales. The adaptive shallow-network
construction of \citet{agapiou2026}, with fixed shifts and random
output coefficients, likewise uses $n$-dependent quantities.

Polynomial density lower bounds make coefficient magnitude contribute
only logarithmically to the negative log-probability of local parameter
intervals. Our uniform and Gaussian constructions satisfy the same
local-mass criterion. Fractional posteriors isolate this
criterion from the global control needed for ordinary posteriors;
clipping then recovers squared prediction-risk bounds.

\Needspace{9\baselineskip}

\section{Statistical model and fractional posterior}
\label{sec:setting}\label{sec:notation}

All functions are assumed measurable. Random objects are defined on
a probability space $(\Omega,\mathcal A,\PP)$, with Borel
$\sigma$-algebras on their state spaces. Fix
$d\geq1$ and write $\Xcal=[0,1]^d$. For an unknown regression
function $f_0:\Xcal\to\mathbb R$, we observe
$\Dcal_n=((X_i,Y_i))_{i=1}^n$ from
\begin{equation}\label{eq:model}
Y_i=f_0(X_i)+\xi_i,\qquad i=1,\ldots,n,
\end{equation}
where the covariates $X,X_i:\Omega\to\Xcal$, $i\geq1$, are
independent with law $P_X$ and the errors $\xi_i:\Omega\to\mathbb R$
are independent standard normal variables, independent of the covariates.
The design law is fixed but otherwise unrestricted unless a density
condition is stated. Write $\E_{f_0}$ and $\PP_{f_0}$ for expectation
and probability under this data law.

Let $\Ccal(\Xcal)$ be the continuous real-valued functions on
$\Xcal$, equipped with the supremum norm and its Borel
$\sigma$-algebra. For a real vector $v=(v_1,\ldots,v_\ell)$
with $\ell\geq1$ coordinates, write
$\norm v_1=\sum_{j=1}^\ell|v_j|$,
$\norm v_2=(\sum_{j=1}^\ell v_j^2)^{1/2}$ and
$\norm v_\infty=\max_{1\leq j\leq\ell}|v_j|$. For a real-valued function $f$
with domain $S$, write $\norm f_\infty=\sup_{x\in S}|f(x)|$.
For $0<\beta\leq1$ and $M>0$, define
\begin{equation}\label{eq:holder}
\Hcal_d^\beta(M)=\bigl\{f\in\Ccal(\Xcal):\norm f_\infty\leq M,
|f(x)-f(y)|\leq M\norm{x-y}_\infty^\beta\text{ for all }x,y\in\Xcal\bigr\}.
\end{equation}
The prediction loss is
$\norm{f-f_0}_{\Lcal^2(P_X)}^2
=\int_\Xcal(f-f_0)^2\,\mathrm dP_X$. Given a prior $\Pi$ on $\Ccal(\Xcal)$ and a fixed
$0<\eta<1$, the fractional posterior is
\begin{equation}\label{eq:posterior}
\Pi_n^\eta(\mathrm d f\mid\Dcal_n)\propto
\exp\!\left\{-\frac\eta2\sum_{i=1}^n
(Y_i-f(X_i))^2\right\}\Pi(\mathrm d f).
\end{equation}
The positive bounded likelihood gives a finite, nonzero normalizing
constant for every data set.
Write $\E_X$ for expectation under $X\sim P_X$.
The one-observation order-$\eta$ R\'enyi divergence in this model is
\begin{equation}\label{eq:renyi}
D_\eta(f,f_0)=-\frac1{1-\eta}\log\E_X
 e^{-\eta(1-\eta)(f(X)-f_0(X))^2/2}.
\end{equation}
We write $a\lesssim b$ if $a\leq Cb$ for a fixed positive constant
$C$, $a\gtrsim b$ if $b\lesssim a$, and $a\asymp b$ if both
inequalities hold. For positive sequences, $a_n\gg b_n$ means
$a_n/b_n\to\infty$. The symbols $\vee$ and $\wedge$ denote maximum
and minimum. Constants in uniform
bounds may depend on $d,\beta,M,\eta$ and the chosen prior, but not
on $n$, the approximation accuracy, or $f_0$ within the stated
H\"older ball. They are independent of $P_X$ unless dependence on
design-density bounds is stated.

\section{Where the prior must place mass}
\label{sec:prior-mass}

Lemma~\ref{lem:frac} specifies how much local prior mass contraction
needs; Proposition~\ref{prop:approx} identifies where to place it.
Both apply before choosing a particular network prior.

\paragraph{How much mass is needed.}
The following form of the fractional-posterior principle of
\citet{bhattacharya2019} converts uniform-approximation mass into
R\'enyi contraction.

\begin{lemma}[Prior mass implies fractional contraction]\label{lem:frac}
Let $\Pi$ be any prior on $\Ccal(\Xcal)$, and let
$\varepsilon_n\downarrow0$ with $n\varepsilon_n^2\to\infty$.
Suppose that, for some fixed $c_0>0$, uniformly over a class of targets,
\[
\Pi\{f:\norm{f-f_0}_\infty\leq\varepsilon_n\}
\geq e^{-c_0n\varepsilon_n^2}.
\]
For $0<\eta<1$ and
$J>(c_0+\eta)/(1-\eta)$,
\begin{equation}\label{eq:frac-general}
\sup_{f_0}\E_{f_0}\Pi_n^\eta
\bigl(D_\eta(f,f_0)>J\varepsilon_n^2\mid\Dcal_n\bigr)
\leq e^{-n\varepsilon_n^2/8}
+e^{-\{(1-\eta)J-c_0-\eta\}n\varepsilon_n^2}.
\end{equation}
\end{lemma}

At $\varepsilon_n\asymp(\log n/n)^{\beta/(2\beta+d)}$,
$n\varepsilon_n^2\asymp
\varepsilon_n^{-d/\beta}\log(1/\varepsilon_n)$.
It therefore suffices that
\[
-\log\Pi\{f:\norm{f-f_0}_\infty\leq\varepsilon\}
\lesssim\varepsilon^{-d/\beta}\log(1/\varepsilon)
\]
for all small $\varepsilon$, uniformly over each $\Hcal_d^\beta(M)$.
The lemma is proved in Appendix~\ref{app:fractional-proof}.

\paragraph{Where the mass is needed.}
For a width $k\geq1$, let
$\theta=((a_m,u_m,b_m))_{m=1}^k\in\mathbb R^{(d+2)k}$, where
$a_m,b_m\in\mathbb R$ are output coefficients and biases, and
$u_m\in\mathbb R^d$ are hidden-layer coefficient vectors.
With the ReLU activation $\rho(t)=\max\{t,0\}$, define
\[
f_\theta(x)=\sum_{m=1}^ka_m\rho(u_m^\top x+b_m),
\qquad x\in\Xcal.
\]
The following is \citet[Cor.~2.4(3)]{yangzhou2025}, with the accuracy
expressed through the width and with zero units added when the
approximating network has fewer than $k$ units. Put
$\omega=(d+3)/(2d)$.

\begin{proposition}[Approximation by shallow networks]
\label{prop:approx}
Fix $0<\beta\leq1$ and $M>0$. There is $C>0$ such that, for every
$f_0\in\Hcal_d^\beta(M)$ and every $k\geq1$, some
$\theta_k^*=((a_m^*,u_m^*,b_m^*))_{m=1}^k$ satisfies
\begin{align}
\norm{f_{\theta_k^*}-f_0}_\infty&\leq Ck^{-\beta/d},
\label{eq:approx-1}\\
\max_{1\leq m\leq k}\{\norm{u_m^*}_\infty,|b_m^*|\}&\leq1,
\label{eq:approx-2}\\
\sum_{m=1}^k|a_m^*|&\leq Ck^{\omega-\beta/d}.
\label{eq:approx-3}
\end{align}
\end{proposition}

Choose $k\asymp\varepsilon^{-d/\beta}$ so the approximator is
$\varepsilon/2$-accurate. The ReLU Lipschitz bound ensures that a
parameter box about $\theta_k^*$ of radius $\delta>0$ polynomially small
in $\varepsilon$ contains only $\varepsilon$-accurate functions.
Let $\mathsf K$ denote the random width and $\theta_k$ the parameter
vector conditional on $\mathsf K=k$. The prior probability of these
functions is therefore at least
\[
\PP(\mathsf K=k)\cdot
\PP\bigl(\norm{\theta_k-\theta_k^*}_\infty\leq\delta
\mid\mathsf K=k\bigr).
\]
It suffices to bound each negative log-probability by
$Ck\log(k/\delta)\asymp k\log(1/\varepsilon)$.
Section~\ref{sec:prior} imposes target-independent conditions giving
these bounds; the approximation and perturbation details are in
Appendices~\ref{app:approx} and~\ref{app:mass}.

\section{Sufficient conditions on the prior}
\label{sec:prior}\label{sec:mechanism}

We now give conditions on the width and network-parameter laws that ensure the
mass required in Section~\ref{sec:prior-mass}, without reference to
$n$, $\beta$ or $M$.

We consider one fixed hierarchical prior on finite-width ReLU
networks. Draw a positive integer width $\mathsf K:\Omega\to\mathbb N$
with mass function $\pi_{\mathsf K}(k)=\PP(\mathsf K=k)$ and then draw
$\theta_k:\Omega\to\mathbb R^{(d+2)k}$, with
$\theta_k=((A_m,U_m,B_m))_{m=1}^k$, from a law
$\Pi_k^\theta$ on $\mathbb R^{(d+2)k}$, conditional on
$\mathsf K=k$. Write $U_m=(U_{m1},\ldots,U_{md})$ for the
hidden-layer coefficient vector. The width and parameters are independent of the
covariates and errors. The random function
$F:\Omega\to\Ccal(\Xcal)$ is
\begin{equation}\label{eq:network}
F(x)=f_{\theta_{\mathsf K}}(x)
=\sum_{m=1}^{\mathsf K}A_m\rho(U_m^\top x+B_m),
\qquad x\in\Xcal.
\end{equation}
Its law $\Pi$ on $\Ccal(\Xcal)$ is well defined because the
parameter-to-function map is continuous at each finite width.
Neither the width law nor the conditional network-parameter laws depend on
$n$, $\beta$, or $M$, and independence between network-parameter coordinates
is not assumed.

\subsection{General sufficient conditions}
\label{sec:prior-conditions}

\paragraph{Width probabilities.}
The width law must not make large widths too improbable. For some
fixed $C_\pi>0$, require
\begin{equation}\label{eq:width}
\log\pi_{\mathsf K}(k)\geq -C_\pi k\log(ek),\qquad k\geq1.
\end{equation}
This permits a log-cost of order $k\log(1/\varepsilon)$ at the
required width. Examples include $\pi_{\mathsf K}(k)=2^{-k}$,
$\mathsf K\sim1+\mathrm{Poisson}(\lambda)$ for fixed $\lambda>0$,
and $\pi_{\mathsf K}(k)\propto k^{-a}$ for fixed $a>1$.
A width law with probabilities proportional to $e^{-k^2}$ is an example
that fails condition~\eqref{eq:width}.

\paragraph{Conditional network-parameter probabilities.}
Since $Ck^{\omega-\beta/d}\leq\tfrac12k^\omega$ eventually for
every $\beta,M$, the approximators satisfy the following bounds at
all large widths. Call $\theta^*=((a_m^*,u_m^*,b_m^*))_{m=1}^k$
an admissible centre if
\begin{equation}\label{eq:weight-centres}
\max_{1\leq m\leq k}\{\norm{u_m^*}_\infty,|b_m^*|\}\leq1,
\qquad \sum_{m=1}^k|a_m^*|\leq\tfrac12 k^\omega.
\end{equation}
For some fixed $C_w>0$, require
\begin{equation}\label{eq:weight-mass}
\begin{split}
-\log\Pi_k^\theta\{\theta:
\norm{\theta-\theta^*}_\infty\leq\delta\}
\leq C_w k\log(ek/\delta),
\end{split}
\end{equation}
for every $k\geq1$, every admissible centre $\theta^*$ at width $k$,
and every $0<\delta\leq1/2$.

Condition~\eqref{eq:weight-mass} requires mass at least
$(\delta/ek)^{C_wk}$ near each admissible centre: a log-cost
proportional to the number of parameters times the logarithm of the
required precision. This is the order obtained from uniform laws on
polynomially growing boxes. Only local mass matters; bounded supports
and dependent coordinates are allowed. Together with
\eqref{eq:width}, it gives the required function-space mass.

\begin{proposition}[Local prior mass]\label{prop:smallball}
Let $\Pi$ satisfy \eqref{eq:width} and \eqref{eq:weight-mass}. For
every $0<\beta\leq1$ and $M>0$ there are $C,\varepsilon_0>0$ such
that, for all $f_0\in\Hcal_d^\beta(M)$ and
$0<\varepsilon<\varepsilon_0$,
\begin{equation}\label{eq:smallball}
\Pi\{f:\norm{f-f_0}_\infty\leq\varepsilon\}
\geq\exp\{-C\varepsilon^{-d/\beta}\log(1/\varepsilon)\}.
\end{equation}
\end{proposition}

The proof is in Appendix~\ref{app:mass}. The conditions are
sufficient, not necessary: only centres supplied by
Proposition~\ref{prop:approx} are used in the proof.

\subsection{Priors satisfying the conditions}\label{sec:examples}

We verify \eqref{eq:weight-mass} for three familiar network-parameter families:
uniform, analogous to the slab component of
\citet{polsonrockova2018}; Gaussian, covered by \citet{kongkim2025};
and Student $t$, including Cauchy, as considered by
\citet{castilloegels2025}. Table~\ref{tab:families} specifies our
conditional laws and width scalings, not those papers' joint priors.
For $a<b$, $\Unif[a,b]$ denotes the uniform distribution on $[a,b]$.
Here $\mathrm N(0,s^2)$ denotes the normal distribution with mean
zero and variance $s^2$, and $t_\nu$ denotes the standard Student $t$
distribution with $\nu>0$ degrees of freedom.

\begin{table}[htbp]
\caption{Conditional network-parameter priors satisfying \eqref{eq:weight-mass}.}
\label{tab:families}
\centering\small
\begin{tabular}{@{}lll@{}}
\toprule
Family & $A_m$ & $U_{mj},B_m$\\
\midrule
Uniform & $\Unif[-k^\omega,k^\omega]$ & $\Unif[-2,2]$\\
Gaussian & $\mathrm N(0,k^{3/d})$ & $\mathrm N(0,1)$\\
Student $t$ & $t_\nu$ & $t_\nu$\\
\bottomrule
\end{tabular}
\end{table}

\begin{lemma}[Uniform, Gaussian and Student network-parameter laws]
\label{lem:families}
Fix $\nu>0$. For each $k\geq1$, let the coordinates $A_m$, $U_{mj}$
and $B_m$ $(1\leq m\leq k,\ 1\leq j\leq d)$ of $\theta_k$ be
independent under $\Pi_k^\theta$, with laws given by one row of
Table~\ref{tab:families}. Then $\Pi_k^\theta$ satisfies
\eqref{eq:weight-mass} with a constant $C_w$ depending only on $d$
for the uniform and Gaussian rows, and only on $d$ and $\nu$ for the
Student row.
\end{lemma}

The families differ in how they reach large output coefficients:
uniform support expands to contain them, Gaussian variance grows to
control their quadratic log-penalty, while Student tails incur only
logarithmic penalties without rescaling ($\nu=1$ gives Cauchy).
Hidden-layer coefficients and biases need no rescaling. Unit-variance Gaussian
output coefficients would fail \eqref{eq:weight-mass}, which allows a
single coefficient of size $\tfrac12k^\omega$.
The Student argument extends to any fixed density with a polynomial
lower bound.

\begin{proposition}[Polynomial density lower bounds]\label{cor:polytail}
Let $q$ be a probability density on $\mathbb R$ such that
\[
q(z)\geq c(1+|z|)^{-r},\qquad z\in\mathbb R,
\]
for some $c>0$ and $r>1$. If, for each $k\geq1$, the coordinates of
$\theta_k$ are independent under $\Pi_k^\theta$, each with density
$q$, then $\Pi_k^\theta$ satisfies \eqref{eq:weight-mass} with a
constant $C_w$ depending only on $d$, $c$ and $r$.
\end{proposition}

This includes Student $t_\nu$ ($r=\nu+1$) and any fixed density
mixture with a positive Cauchy component.
Paired with any width law satisfying \eqref{eq:width}, these give
priors covered by the results below. The probability estimates are
in Appendix~\ref{app:weight-boxes}.

\section{Contraction and prediction risk}\label{sec:results}

Lemma~\ref{lem:frac} and Proposition~\ref{prop:smallball} give the
near-minimax rate derived in Section~\ref{sec:prior-mass}.
Theorem~\ref{cor:risk} states the resulting
posterior-mean risk; Theorem~\ref{thm:main} provides the underlying
contraction bounds.

To pass from R\'enyi divergence to squared prediction loss, we need
bounded functions. Write $C_b(t)=(-b)\vee(t\wedge b)$ for clipping
at $b>0$. The data-based choice below avoids a known target bound.

\begin{theorem}[Near-minimax posterior-mean risk]\label{cor:risk}
Let $\Pi$ be a fixed hierarchical prior satisfying
\eqref{eq:width} and \eqref{eq:weight-mass}, and fix $0<\eta<1$.
There is a sequence of finite positive data-dependent cutoffs
$\widetilde B_n=\widetilde B_n(\Dcal_n)$, $n\geq3$, depending
only on $n,d$ and the data, with the following property. For every
$0<\beta\leq1$, $M>0$, and design law $P_X$ with Lebesgue density
$p_X$ satisfying, for fixed $0<c_X\leq C_X<\infty$ and
Lebesgue-almost every $x\in\Xcal$,
\begin{equation}\label{eq:design}
0<c_X\leq p_X(x)\leq C_X<\infty,
\end{equation}
the posterior means
\begin{equation}\label{eq:adaptive-mean}
\widetilde f_n(x)=\int_{\Ccal(\Xcal)}
 C_{\widetilde B_n}(f(x))\,
 \Pi_n^\eta(\mathrm d f\mid\Dcal_n)
\end{equation}
satisfy, for all sufficiently large $n$,
\begin{equation}\label{eq:risk}
\begin{split}
\sup_{f_0\in\Hcal_d^\beta(M)}\E_{f_0}
\norm{\widetilde f_n-f_0}_{\Lcal^2(P_X)}^2
\lesssim(\log n/n)^{2\beta/(2\beta+d)}.
\end{split}
\end{equation}
The upper bound requires only the lower density bound;
under both bounds in \eqref{eq:design},
\begin{equation}\label{eq:lower}
\begin{split}
\inf_{f_n}\sup_{f_0\in\Hcal_d^\beta(M)}\E_{f_0}
\norm{f_n-f_0}_{\Lcal^2(P_X)}^2
\gtrsim n^{-2\beta/(2\beta+d)},
\end{split}
\end{equation}
where the infimum is over all measurable estimators. Constants may
depend on $c_X$ and $C_X$.
\end{theorem}

The same estimator adapts to every stated H\"older ball and design
law, losing only $(\log n)^{2\beta/(2\beta+d)}$ relative to the
classical minimax bound \citep{stone1982}; see
Section~\ref{sec:log-gap}. Its upper risk bound follows from the
next contraction result.

\begin{theorem}[Fractional-posterior contraction]\label{thm:main}
Let $\Pi$ be a fixed hierarchical prior satisfying
\eqref{eq:width} and \eqref{eq:weight-mass}. Fix $0<\beta\leq1$,
$M>0$, and $0<\eta<1$. There are $A,C,c>0$ such that, with
\begin{equation}\label{eq:rate}
\varepsilon_n=A(\log n/n)^{\beta/(2\beta+d)},\qquad n\geq3,
\end{equation}
the following hold for all sufficiently large $n$.

\textup{(i)} \emph{R\'enyi contraction.} For every design law $P_X$,
\begin{equation}\label{eq:renyi-result}
\sup_{f_0\in\Hcal_d^\beta(M)}\E_{f_0}\Pi_n^\eta
\bigl(D_\eta(f,f_0)>C\varepsilon_n^2\mid\Dcal_n\bigr)
\leq2e^{-cn\varepsilon_n^2}.
\end{equation}

\textup{(ii)} \emph{$\Lcal^2$ contraction after clipping.} The
cutoffs $\widetilde B_n$ of Theorem~\ref{cor:risk} can be chosen so
that, in addition, the image $\widetilde\Pi_n^\eta$ of
$\Pi_n^\eta(\cdot\mid\Dcal_n)$ under
$f\mapsto C_{\widetilde B_n}\circ f$ satisfies, whenever $P_X$ has a density $p_X\geq c_X>0$
Lebesgue-almost everywhere,
\begin{equation}\label{eq:l2-result}
\sup_{f_0\in\Hcal_d^\beta(M)}\E_{f_0}\widetilde\Pi_n^\eta
\bigl(\norm{g-f_0}_{\Lcal^2(P_X)}^2>C'\varepsilon_n^2
\mid\Dcal_n\bigr)\leq C'e^{-c'n\varepsilon_n^2}
\end{equation}
for constants $C',c'>0$ that may also depend on $c_X$.

The prior is the same for every $n$, $\beta$ and $M$, and the cutoff
rule uses only $n$, $d$ and the data.
\end{theorem}

\begin{remark}[Computing the cutoff]\label{rem:clipping-construction}
For $n\geq3$, put $m_n=\lceil\log n\rceil$ and partition
$\Xcal$ into $J_n=m_n^d$ cubes $\mathcal Q_{j,n}$ of side
$1/m_n$, using half-open intervals with the outer faces included.
Let $N_{j,n}$ be the number of observations in cube $j$ and
$\overline Y_{j,n}$ their response mean, set to zero when
$N_{j,n}=0$, and define
\begin{equation}\label{eq:adaptive-clipping}
\widetilde B_n=2+\max_{1\leq j\leq J_n}|\overline Y_{j,n}|.
\end{equation}
Counts and response sums give this cutoff in one pass and
$O(dn+J_n)$ operations, without sample splitting or knowledge of
$\beta$, $M$ or $P_X$. The margin $2$ absorbs cell-mean noise and
within-cell variation. Under the lower density bound, for all large
$n$,
\[
\sup_{f_0\in\Hcal_d^\beta(M)}
\PP_{f_0}\{\norm{f_0}_\infty>\widetilde B_n
\text{ or }\widetilde B_n>M+3\}\leq CJ_ne^{-cn/J_n}.
\]
On the complementary event, clipping cannot increase R\'enyi
divergence, and boundedness lets it control squared
$\Lcal^2(P_X)$ loss.
\end{remark}

\begin{remark}[The clipped draws remain ReLU networks]
For every $b>0$,
\[
C_b(t)=\rho(t+b)-\rho(t-b)-b.
\]
Thus each clipped draw is an exact finite ReLU network, in general
with one additional hidden layer; the posterior mean need not be.
\end{remark}

\begin{remark}[Arbitrary design laws]\label{rem:arbitrary-design}
There is another finite positive data-dependent cutoff
$\widetilde B_n^\circ$, again using only $n,d$ and the data, for
which every $P_X$ satisfies
\begin{equation}\label{eq:designfree-l2}
\sup_{f_0\in\Hcal_d^\beta(M)}\E_{f_0}\Pi_n^\eta
\bigl(\norm{C_{\widetilde B_n^\circ}\circ f-f_0}_{\Lcal^2(P_X)}^2
>C\varepsilon_n^2\mid\Dcal_n\bigr)
\leq Ce^{-cn\varepsilon_n^2}+Cn^{-15/4}.
\end{equation}
The posterior mean retains the upper risk rate in \eqref{eq:risk}.
The construction in Appendix~\ref{app:arbitrary-design} controls
clipping bias instead of covering the entire target range.
With a known $B\geq M$, deterministic clipping at $B$ gives a purely
exponential tail for every $P_X$.
\end{remark}

Proofs are in Appendices~\ref{app:posterior} and~\ref{app:minimax}.
Appendix~\ref{app:bounded-support} extends the R\'enyi and
arbitrary-design clipping results to bounded Borel covariate supports,
using ambient H\"older continuity in $\mathbb R^d$.

\setlength{\dbltextfloatsep}{6pt plus 2pt minus 2pt}

\section{Numerical assessment}
\label{sec:simulation}

\subsection{Simulation setup}
For $d=1$, $P_X=\Unif[0,1]$, Gaussian errors with variance one, and
$\eta=0.9$, we assess the squared prediction error in
{Theorem~\ref{cor:risk}} at the fixed targets
{
\[
f_\beta(x)=\abs{\sin(2\pi x)}^\beta,
\qquad \beta\in\{1/2,1\}.
\]
Their global H\"older bounds are $\sqrt{2\pi}$ and $2\pi$.
Both have cusps at $1/2$ and peaks at $1/4$ and $3/4$.
Neither has an exact finite shallow ReLU representation;
see Appendix~\ref{app:simulation} for verification of these facts.}

For each target, sample size and parameter family, we use 50
replications and retain the first computation stage that passes the
diagnostics in Appendix~\ref{app:simulation-computation}.
The sample sizes are $n\in\{2^j:j=6,\ldots,12\}$.

The baseline (N/G) uses the Gaussian row of
Table~\ref{tab:families} and geometric width $\pi_G(k)=2^{-k}$.
We estimate $\widetilde B_n$ from the fitting data and average
clipped draws as in \eqref{eq:adaptive-mean}. Both targets are bounded
by one and $\widetilde B_n\geq2$.
{Computation makes the simplification of restricting the width to
$\mathsf K\leq64$ and uses a 1009-point grid of values for numerical integration.}

{Figure~\ref{fig:lipschitz-credible} shows a H\"older $1/2$ fit and a Lipschitz fit at each displayed sample size under the uniform parameter prior. The true central-notch value zero lies in none of the displayed pointwise intervals, although we note that these fractional-posterior intervals carry no frequentist coverage guarantee.}

\begin{figure}[t]
\centering
\includegraphics[width=0.9\textwidth]{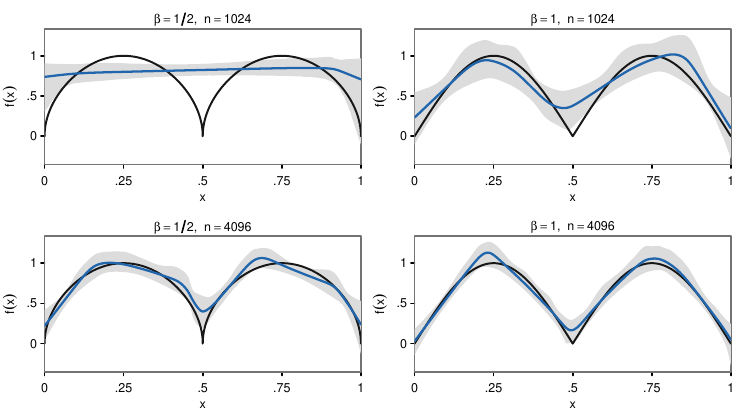}
\caption{Example fits of the H\"older $1/2$ case (left) and Lipschitz case (right), at $n=1024$ (top) and $4096$ (bottom), using the uniform parameter prior with geometric width and $\eta=0.9$. Truth (black), clipped posterior mean (blue), and pointwise 95\% credible intervals (grey).}
\label{fig:lipschitz-credible}
\end{figure}

\subsection{Rate and prior sensitivity}
{Figure~\ref{fig:adaptive-risk} reports mean integrated
squared errors with 95\% normal Monte Carlo confidence intervals
across the 50 replications per case (Appendix~\ref{app:simulation-adaptive-risk}).
The grey guides show $q_{n,\beta}=(\log n/n)^{2\beta/(2\beta+1)}$
(upper) and $n^{-2\beta/(2\beta+1)}$ (lower), both matched to
N/G at $n=64$ with illustrative constants.}

\begin{figure}[!t]
\centering
\includegraphics[width=\textwidth]{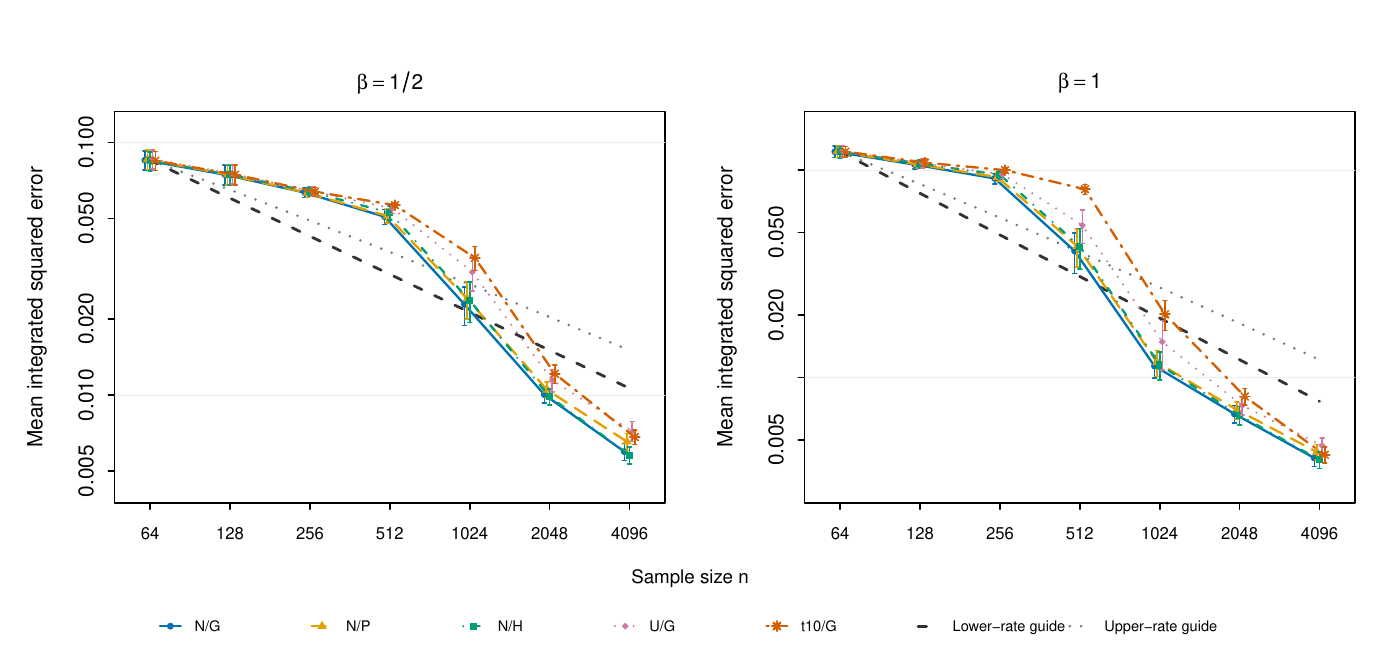}
\caption{{Mean squared errors with 95\% Monte Carlo
confidence intervals from 50 replications per case. N/U/$t_{10}$:
Gaussian/uniform/Student $t_{10}$ parameters; G/P/H:
geometric/shifted Poisson/power-law widths. Grey illustrative
guides: lower (dashed), upper (dotted).}}
\label{fig:adaptive-risk}
\end{figure}

{At $n=4096$, N/G mean errors are $0.00597$ for $\beta=1/2$ and $0.00408$ for $\beta=1$. U/G and $t_{10}$/G mean errors are $0.00719$ and $0.00681$ for $\beta=1/2$, and $0.00469$ and $0.00423$ for $\beta=1$. Across sample sizes and both reweighted Gaussian width laws, the mean errors differ from N/G by at most $9.5\%$ for $\beta=1/2$ and $6.9\%$ for $\beta=1$. For N/G, from $2048$ to $4096$, the estimated decay exponents are $0.749$ for $\beta=1/2$ and $0.707$ for $\beta=1$, with estimated 95\% paired-bootstrap intervals $[0.630,0.879]$ and $[0.575,0.839]$, respectively.
These intervals are obtained by resampling the 50 replications with
replacement, using the same sampled indices at both sample sizes, and
recalculating the decay exponent \citep{efron1985bootstrap}. They quantify simulation uncertainty
while preserving the dependence between errors calculated from the same
underlying data stream; Appendix~\ref{app:simulation-adaptive-risk}
gives details.}

Posterior samples obtained under Gaussian network parameters and
geometric width are reweighted to estimate the posterior mean under fixed
shifted Poisson and power-law width priors satisfying \eqref{eq:width},
conditional on $\mathsf K\leq64$. Each case uses 50 replications whose fits pass the diagnostics
in Appendix~\ref{app:simulation-computation}.

\Needspace{8\baselineskip}
\section{Discussion}

The main conclusions depend on two local probability conditions:
one on the width law and one on the conditional network-parameter law. A single
prior satisfying them gives the stated rate for every H\"older
smoothness and radius under consideration. Uniform, Gaussian, and
Student network-parameter priors are distinct examples of this principle, rather than
separate mechanisms. Data-dependent clipping then gives the same
posterior-mean risk rate without a known target bound.

\subsection{The logarithmic gap}\label{sec:log-gap}

The local prior-mass bound \eqref{eq:smallball} has a
$\log(1/\varepsilon)$ factor, while the minimax lower bound has
no logarithm. Requiring the network parameters to lie in narrow intervals
around one approximator contributes this factor, and the width law
can contribute a cost of the same order. These estimates do not
show that the logarithmic loss is necessary.

A uniform bound
$-\log\Pi\{f:\norm{f-f_0}_\infty\leq\varepsilon\}
\lesssim\varepsilon^{-d/\beta}$ would give the exact minimax rate
through Lemma~\ref{lem:frac}. A heavier-tailed width law alone does
not remove the interval-probability cost. Possible improvements
include collecting the probability of many accurate representations,
using dependence between network parameters, and working with likelihood or
prediction neighbourhoods rather than a single uniform-approximation
box. The sufficient condition \eqref{eq:weight-mass} is designed for
transparent verification, not as a characterization of all priors
that attain the rate.

\subsection{Deep networks and other approximations}

Using shallow networks keeps our approach simple, with explicit
architecture and perturbation estimates. The method can also be
considered for deep networks. For deep
networks, the corresponding problem is to specify one prior on
finite architectures and, conditional on each architecture, assign
enough mass near accurate parameter configurations. The relevant
approximation and stability bounds must track depth and parameter
count, as well as coefficient size. Combining such bounds with a
fixed architecture law would extend the same principle without
choosing a new prior for each sample size or smoothness class.
The resulting rates must be established from those bounds; they do
not follow merely by replacing width with total parameter count.

The prior-mass principle is also not specific to neural networks.
Lemma~\ref{lem:frac} applies to any prior on $\Ccal(\Xcal)$ in the
Gaussian regression model. Finite series, spline expansions, and
other finite-dimensional approximation families offer settings in
which to formulate analogous conditions on model size and local
parameter probabilities. The objective is to identify verifiable
conditions on the approximation family and prior that guarantee
enough mass across all relevant scales. Outside Gaussian regression,
the corresponding likelihood-neighbourhood and loss-transfer
arguments would also need to be established; the clipping result
here is a separate regression-specific step.

\subsection{Further questions}

Our results concern a fractional posterior with $0<\eta<1$.
Its expected fractional likelihood ratio decays exponentially in R\'enyi
divergence. At $\eta=1$ the expected likelihood ratio is one, so
this particular argument gives no contraction bound. An
ordinary-posterior result would require additional control of
inaccurate networks, including very large architectures; compare
\citet{kongkim2025}. The local conditions in this paper do not
supply that global control.

We treat H\"older smoothness up to one. For
$1<\beta<(d+3)/2$, the shallow approximation bounds of
\citet{yangzhou2025} offer a starting point, with smoothness defined
through derivatives and an appropriate extension beyond the cube.
The clipping construction could then use the resulting Lipschitz
bound. Deep-network results for smoother classes are available in
\citet{kongkim2025}.

Adaptive clipping removes the need to know the target range, but
at small sample sizes its interval may fail to cover the whole
range. A lower design-density bound gives eventual exponential
control of this event. Remark~\ref{rem:arbitrary-design} retains
the same risk rate without that assumption, at the cost of an
additional polynomial term in the posterior-tail bound. The present
analysis does not determine the risk of the unclipped posterior
mean, when it exists, or frequentist coverage of posterior credible balls.

The design law can permit faster rates than the uniform
H\"older-ball upper bound. When the covariates equal one point,
only one function value needs to be estimated and squared risk has
order $n^{-1}$. Rates for designs concentrated on lower-dimensional
sets require geometric information about the support;
\citet[Section~3.3]{castilloegels2025} study deep-network rates in
terms of its upper Minkowski dimension.

The numerical width cap excludes larger networks, while finite
posterior sampling adds computational error. Bounds keeping these
errors below the statistical risk would connect the theory more
directly to implementation. Cauchy priors also motivate samplers that
explore large coefficients reliably. Extending the variational
analysis of \citet{castilloegels2025} to a fixed random-architecture
prior is another route to studying computational approximations.

\subsection*{AI Use Statement}
OpenAI ChatGPT and Codex were used for exposition, mathematical
formulation, proof checking, proofreading, and formatting.
The authors take responsibility for all parts of this manuscript.

% References are integrated to make this source independent of BibTeX.

\clearpage
\appendix
% Appendices are integrated into the article.
\section{Approximation theory}\label{app:approx}

\begin{lemma}[H\"older extension]\label{lem:extension}
Let $A$ be a nonempty subset of a metric space $(S,d_S)$ and let
$0<\beta\leq1$. If $h:A\to\mathbb R$ satisfies
$\abs{h(a)-h(a')}\leq H d_S(a,a')^\beta$ for all $a,a'\in A$
and some $H>0$, then
\begin{equation}\label{eq:extension}
\widetilde h(z)=\inf_{a\in A}\{h(a)+H d_S(z,a)^\beta\},
\qquad z\in S,
\end{equation}
is a finite extension of $h$ with the same H\"older constant. If $h$ is
bounded by $M$, clipping $\widetilde h$ to $[-M,M]$ preserves both the
extension and its H\"older bound.
\end{lemma}

\begingroup
\begin{proof}
For $z,z'\in S$ and $a\in A$, subadditivity of
$t\mapsto t^\beta$ gives
\[
d_S(z,a)^\beta\leq d_S(z,z')^\beta+d_S(z',a)^\beta.
\]
We first check that the infimum in \eqref{eq:extension} is finite.
Fix $a_0\in A$. The H\"older condition and the same subadditivity give,
for every $a\in A$,
\[
h(a)+H d_S(z,a)^\beta
\geq h(a_0)-H d_S(z,a_0)^\beta.
\]
This bounds the infimum below; taking $a=a_0$ bounds it above.

If $z\in A$, the choice $a=z$ gives $\widetilde h(z)\leq h(z)$.
The H\"older condition gives $h(a)+H d_S(z,a)^\beta\geq h(z)$
for every $a\in A$, proving equality. For arbitrary $z,z'\in S$,
subadditivity gives
\[
h(a)+H d_S(z,a)^\beta
\leq h(a)+H d_S(z',a)^\beta+H d_S(z,z')^\beta.
\]
Taking infima and then interchanging $z,z'$ proves the asserted
H\"older bound. Finally, projection onto $[-M,M]$ is one-Lipschitz
and fixes every value of $h$, so clipping preserves both conclusions.
\end{proof}
\endgroup

\begingroup
\begin{proof}[Proof of Proposition~\ref{prop:approx}]
We first place the target on the unit ball used by the approximation
result, then rescale the approximating network back to $\Xcal$.
Let $\mathbb B^d$ be the closed Euclidean unit ball in $\mathbb R^d$
and $\mathbb S^d$ the Euclidean unit sphere in $\mathbb R^{d+1}$.
Set $R=\max\{1,\sqrt d\}$ and
$\Xcal_R=R^{-1}\Xcal\subseteq\mathbb B^d$.
The function $h(z)=f_0(Rz)$ on $\Xcal_R$ satisfies
$\norm h_\infty\leq M$ and
\[
\abs{h(z)-h(z')}\leq MR^\beta\norm{z-z'}_2^\beta.
\]
Lemma~\ref{lem:extension}, followed by clipping to $[-M,M]$,
extends $h$ to $\mathbb R^d$ with these same bounds.
There is consequently a fixed $K_0\geq1$, depending only on
$d,\beta,M$, such that the restriction of $h/K_0$ to $\mathbb B^d$
belongs to the unit H\"older ball in
\citet[Cor.~2.4(3)]{yangzhou2025}.

For every positive integer $N$ and bound $V>0$ on the sum of absolute output coefficients, that
corollary supplies a network
\[
g_0(z)=\sum_{m=1}^N c_m\rho((z,1)^\top v_m),
\qquad v_m\in\mathbb S^d,\qquad \sum_{m=1}^N\abs{c_m}\leq V,
\]
with supremum error relative to $h/K_0$ on $\mathbb B^d$ at most
\[
C\left\{N^{-\beta/d}\vee
V^{-2\beta/(d+3-2\beta)}\right\}.
\]
Its smoothness condition holds because
$0<\beta\leq1<(d+3)/2$. Take $N=k$ and
$V=k^{\omega-\beta/d}$. Since
\[
\frac{2\beta(\omega-\beta/d)}{d+3-2\beta}=\frac\beta d,
\]
both terms in the error bound have order $k^{-\beta/d}$.

Write $v_m=(w_m,t_m)\in\mathbb R^d\times\mathbb R$ and set
\[
a_m^*=K_0c_m,\qquad u_m^*=w_m/R,\qquad b_m^*=t_m.
\]
Then $f_{\theta_k^*}(x)=K_0g_0(x/R)$ has supremum error at most
$Ck^{-\beta/d}$ on $\Xcal$ and a sum of absolute output coefficients at
most $K_0k^{\omega-\beta/d}$. Units with zero parameters may be
added if the representation uses fewer than $k$ units.
The sphere constraint and $R\geq1$ imply
$\norm{u_m^*}_\infty\leq1$ and $\abs{b_m^*}\leq1$.
Every constant used above is uniform over the stated H\"older ball.
\end{proof}
\endgroup

\section{Prior probability of approximating networks}\label{app:mass}

Fix $0<\beta\leq1$ and $M>0$, and set
\begin{equation}\label{eq:exponents}
p=d/\beta,\qquad \kappa=(d+3-2\beta)/(2\beta).
\end{equation}
Choose $C_N$ sufficiently large and put
$k_\varepsilon=\lceil C_N\varepsilon^{-p}\rceil$.
Proposition~\ref{prop:approx} gives an approximating network
$f_\varepsilon=f_{\theta_{k_\varepsilon}^*}$ with error at most
$\varepsilon/4$ and sum of absolute output coefficients
$V_\varepsilon=\sum_{m=1}^{k_\varepsilon}|a_m^*|
\lesssim\varepsilon^{-\kappa}$.
For a fixed $c_\delta>0$ to be chosen below, define
\begin{equation}\label{eq:delta}
\delta_\varepsilon=c_\delta\varepsilon/
(k_\varepsilon+1+V_\varepsilon),
\end{equation}
and the parameter box
\begin{equation}\label{eq:event}
\mathcal E_\varepsilon=
\{\theta\in\mathbb R^{(d+2)k_\varepsilon}:
\norm{\theta-\theta_{k_\varepsilon}^*}_\infty
\leq\delta_\varepsilon\}.
\end{equation}
All constants below are uniform over the target H\"older ball.

\begin{lemma}[Box stability]\label{lem:stability}
For sufficiently small fixed $c_\delta>0$ and all sufficiently
small $\varepsilon$, every parameter vector in
$\mathcal E_\varepsilon$ represents a function $f$ satisfying
$\norm{f-f_\varepsilon}_\infty\leq\varepsilon/4$ and
$\norm{f-f_0}_\infty\leq\varepsilon/2$. Moreover,
$\log(1/\delta_\varepsilon)\lesssim\log(1/\varepsilon)$.
\end{lemma}

\begingroup
\begin{proof}[Proof of Lemma~\ref{lem:stability}]
Put $k=k_\varepsilon$ and $\delta=\delta_\varepsilon$.
Fix $\theta=((a_m,u_m,b_m))_{m=1}^k\in\mathcal E_\varepsilon$
and let $f$ be its network. Put $D=d+1$ and abbreviate
$V=V_\varepsilon$. For each $x\in\Xcal$,
\[
\abs{(u_m-u_m^*)^\top x+b_m-b_m^*}\leq D\delta.
\]
For sufficiently small $\varepsilon$, we have $\delta\leq1$.
Since the approximating {hidden-layer coefficients and biases} are bounded by one,
$\abs{\rho(u_m^\top x+b_m)}\leq2D$.
Subtract and add $a_m^*\rho(u_m^\top x+b_m)$ to separate
the perturbation of the output coefficient from that of the hidden-layer
coefficients and bias. ReLU is
one-Lipschitz, so
\begin{align*}
&\abs{a_m\rho(u_m^\top x+b_m)
-a_m^*\rho((u_m^*)^\top x+b_m^*)}\\
&\qquad\leq
\abs{a_m-a_m^*}\abs{\rho(u_m^\top x+b_m)}
+\abs{a_m^*}D\delta
\leq2D\delta+D\abs{a_m^*}\delta.
\end{align*}
Summing these inequalities gives
\[
\norm{f-f_\varepsilon}_\infty
\leq D\delta(2k+V)
\leq2Dc_\delta\varepsilon.
\]
Taking $c_\delta\leq(8D)^{-1}$ proves the first error bound.
Proposition~\ref{prop:approx} and the triangle inequality give
$\norm{f-f_0}_\infty\leq\varepsilon/2$.

Finally, $k+1+V\lesssim\varepsilon^{-\max\{p,\kappa\}}$.
Thus $\delta\gtrsim\varepsilon^{1+\max\{p,\kappa\}}$, and taking
logarithms proves the remaining assertion. These estimates and
choices of constants are uniform over the target class.
\end{proof}
\endgroup

For sufficiently small $\varepsilon$, the centre is admissible in
\eqref{eq:weight-centres}, because
$V_\varepsilon\leq Ck_\varepsilon^{\omega-\beta/d}
\leq k_\varepsilon^\omega/2$. Also
$\delta_\varepsilon\leq1/2$. Condition~\eqref{eq:weight-mass}
therefore gives
\begin{equation}\label{eq:mass-bound}
-\log\Pi_{k_\varepsilon}^{\theta}(\mathcal E_\varepsilon)
\lesssim k_\varepsilon\log(1/\varepsilon).
\end{equation}

\begingroup
\begin{proof}[Proof of Proposition~\ref{prop:smallball}]
If the sampled width is $k_\varepsilon$ and every network parameter
lies in its interval in $\mathcal E_\varepsilon$, Lemma~\ref{lem:stability}
gives $\norm{F-f_0}_\infty\leq\varepsilon/2$ for the resulting
network $F$. Thus
$\{\mathsf K=k_\varepsilon,\theta_{k_\varepsilon}\in\mathcal E_\varepsilon\}$
is contained in $\{\norm{F-f_0}_\infty\leq\varepsilon\}$.
The probability of choosing this width and these parameter values is
$\pi_{\mathsf K}(k_\varepsilon)
\Pi_{k_\varepsilon}^{\theta}(\mathcal E_\varepsilon)$. Therefore
\begin{align*}
-\log\Pi\{f:\norm{f-f_0}_\infty\leq\varepsilon\}
&\leq-\log\pi_{\mathsf K}(k_\varepsilon)
-\log\Pi_{k_\varepsilon}^{\theta}(\mathcal E_\varepsilon)\\
&\lesssim k_\varepsilon\log(1/\varepsilon)
\lesssim\varepsilon^{-d/\beta}\log(1/\varepsilon).
\end{align*}
The second inequality uses \eqref{eq:width} and
\eqref{eq:mass-bound}; the last uses the chosen width
$k_\varepsilon\asymp\varepsilon^{-d/\beta}$.
All constants are uniform over $f_0\in\Hcal_d^\beta(M)$.
\end{proof}
\endgroup

\section{Fractional posterior and clipped loss}\label{app:posterior}

{The clipping level $\widetilde B_n$ is a finite measurable
function of the data. The map $(b,f)\mapsto C_b\circ f$ is continuous
for $b>0$ and $f\in\Ccal(\Xcal)$, so the clipped posterior
$\widetilde\Pi_n^\eta$ is a measurable probability kernel.
For each data set, the integrand in \eqref{eq:adaptive-mean} is
bounded by $\widetilde B_n$. Dominated convergence therefore makes
$\widetilde f_n$ continuous in $x$. Its values at a countable dense
subset of $\Xcal$ are measurable functions of the data, which
makes $\widetilde f_n$ a measurable $\Ccal(\Xcal)$-valued estimator.}

\paragraph{Existence of the unclipped posterior mean.}
\label{app:unclipped-mean}
Take $d=1$, $P_X=\Unif[0,1]$, $f_0=0$, and independent
Student $t_\nu$ network parameters with $0<\nu\leq1$, including
Cauchy network parameters at $\nu=1$. Condition~\eqref{eq:width} gives
$\pi_{\mathsf K}(1)>0$. Fix a realized data set with
$m=\max_{i\leq n}X_i<1$ and finite responses; this holds almost
surely at every finite sample size. Set
$t_0=(m+1)/2$, $I=(t_0,1]$, and
\[
\mathcal H_m=\{1<U_1<2,\ m<-B_1/U_1<t_0\}.
\]
This event has positive prior probability conditional on
$\mathsf K=1$. On $\{\mathsf K=1\}\cap\mathcal H_m$,
all training activations vanish, whereas
$\rho(U_1x+B_1)\geq x-t_0>0$ for $x\in I$.
The fractional likelihood is therefore
$L_\eta=\exp\{-\eta\sum_{i=1}^nY_i^2/2\}>0$, independently of $A_1$.
Let $Z_\eta\in(0,1]$ be the normalizing constant in
\eqref{eq:posterior}, and write $z_+=\max\{z,0\}$ and
$z_-=\max\{-z,0\}$. The Student density has tails of order
$|a|^{-\nu-1}$, so both
$\E[(A_1)_\pm\mid\mathsf K=1]$ are infinite. Independence and
Tonelli's theorem give, for every $x\in I$,
\begin{align*}
\int_{\Ccal(\Xcal)} f(x)_\pm\,\Pi_n^\eta(\mathrm d f\mid\Dcal_n)
&\geq \frac{\pi_{\mathsf K}(1)L_\eta}{Z_\eta}
(x-t_0)\PP(\mathcal H_m\mid\mathsf K=1)\\
&\qquad{}\times\E[(A_1)_\pm\mid\mathsf K=1]=\infty.
\end{align*}
Thus the unclipped posterior mean is undefined on an interval of
positive $P_X$-measure. This concerns the mean's existence, rather
than posterior contraction in probability. The clipped averages
remain well defined without any moment condition on the network parameters.

\begin{lemma}[Clipping with target-range bias]\label{lem:clipping}
For any $b>0$, $f,f_0\in\Ccal(\Xcal)$, and $0<\eta<1$,
\begin{equation}\label{eq:adaptive-clipping-lemma}
\norm{C_b\circ f-f_0}_{\Lcal^2(P_X)}^2
\leq 2K_\eta(b)D_\eta(f,f_0)
+2\norm{C_b\circ f_0-f_0}_{\Lcal^2(P_X)}^2,
\end{equation}
where $K_\eta(b)=4(1-\eta)b^2/
\{1-e^{-2\eta(1-\eta)b^2}\}$.
If $b\geq\norm{f_0}_\infty$, the second term vanishes and the
factor $2$ in front of the first term can be omitted.
\end{lemma}

\begingroup
\begin{proof}
Set $a=\eta(1-\eta)/2$ and $T=(f(X)-f_0(X))^2$.
Clipping is one-Lipschitz, and both clipped values lie in $[-b,b]$.
Hence
\[
|C_b(f(X))-C_b(f_0(X))|^2\leq\min\{T,4b^2\}.
\]
On $[0,4b^2]$, the concave function $t\mapsto1-e^{-at}$ lies
above the chord joining its endpoints. For $t\geq4b^2$ it is
increasing. Combining these two facts gives
\begin{equation}\label{eq:clipping-chord}
1-e^{-aT}\geq
\frac{1-e^{-4ab^2}}{4b^2}\min\{T,4b^2\}.
\end{equation}
By \eqref{eq:renyi} and $-\log z\geq1-z$,
\[
(1-\eta)D_\eta(f,f_0)
\geq\E_X(1-e^{-aT}).
\]
Taking expectations in \eqref{eq:clipping-chord} therefore yields
\[
\norm{C_b\circ f-C_b\circ f_0}_{\Lcal^2(P_X)}^2
\leq K_\eta(b)D_\eta(f,f_0).
\]
Now write $C_b\circ f-f_0$ as the sum of
$C_b\circ f-C_b\circ f_0$ and $C_b\circ f_0-f_0$.
The inequality $(u+v)^2\leq2u^2+2v^2$, integrated over $P_X$,
proves \eqref{eq:adaptive-clipping-lemma}. If $b\geq\norm{f_0}_\infty$,
then $C_b\circ f_0=f_0$, so this final step and its factor two are
unnecessary.
\end{proof}
\endgroup

\begingroup
\begin{proof}[Proof of Theorem~\ref{thm:main}]
The prior-probability bound \eqref{eq:smallball} is uniform over
$\Hcal_d^\beta(M)$. For $\varepsilon_n$ in \eqref{eq:rate},
\[
\varepsilon_n^{-d/\beta}\log(1/\varepsilon_n)
\lesssim A^{-d/\beta}
n^{d/(2\beta+d)}(\log n)^{2\beta/(2\beta+d)},
\]
whereas
$n\varepsilon_n^2=A^2n^{d/(2\beta+d)}
(\log n)^{2\beta/(2\beta+d)}$.
Choose the fixed $A$ large enough that the prior probability is at
least $e^{-n\varepsilon_n^2}$ for all sufficiently large $n$.
Since $n\varepsilon_n^2\to\infty$, Lemma~\ref{lem:frac} applies with
$c_0=1$. Choose $J_0>(1+\eta)/(1-\eta)$ and put
\[
C=J_0,\qquad
c=\min\{1/8,(1-\eta)J_0-1-\eta\}>0.
\]
The two terms in that lemma are at most $e^{-cn\varepsilon_n^2}$
each, proving \eqref{eq:renyi-result}.

Write $p_j=P_X(\mathcal Q_{j,n})$. Equal cell volumes and the lower
density bound imply $p_j\geq c_X/J_n$.
Let $\mathcal E_n$ be the event that every cell contains at least
$c_Xn/(2J_n)$ observations. The binomial Chernoff bound and a union
bound over cells give
\[
\PP(\mathcal E_n^c)\leq J_n\exp\{-c_Xn/(8J_n)\}.
\]
Let $\overline\xi_{j,n}$ denote the noise mean in cell $j$, with
value zero if the cell is empty. Conditional on the covariates,
each occupied-cell noise mean has law $\mathrm N(0,1/N_{j,n})$.
Consequently, on $\mathcal E_n$,
\[
\PP\{\max_j|\overline\xi_{j,n}|>1\mid X_1,\ldots,X_n\}
\leq2J_n\exp\{-c_Xn/(4J_n)\}.
\]
Define $\mathcal G_n=\mathcal E_n\cap
\{\max_j|\overline\xi_{j,n}|\leq1\}$. Combining the two failure
probabilities gives
\begin{equation}\label{eq:coarse-good-bound}
\PP_{f_0}(\mathcal G_n^c)
\leq3J_n\exp\{-c_Xn/(8J_n)\}.
\end{equation}
This estimate is uniform over the target class. Since
$J_n\asymp(\log n)^d$,
\[
\frac{n/J_n}{n\varepsilon_n^2}
\asymp A^{-2}n^{2\beta/(2\beta+d)}
(\log n)^{-d-2\beta/(2\beta+d)}\longrightarrow\infty.
\]
Also $\log J_n=o(n/J_n)$. Thus the right side of
\eqref{eq:coarse-good-bound} is eventually at most
$e^{-cn\varepsilon_n^2}$ for the same fixed $c$ chosen above.

For large enough $n$, $Mm_n^{-\beta}\leq1$.
On $\mathcal G_n$, every cell is occupied. Its max-norm diameter is
$1/m_n$, so for $x\in\mathcal Q_{j,n}$, averaging the H\"older
inequality over the sampled covariates in that cell gives
\begin{equation}\label{eq:coarse-range}
|f_0(x)|\leq|\overline Y_{j,n}|+
|\overline\xi_{j,n}|+Mm_n^{-\beta}
\leq\widetilde B_n.
\end{equation}
In addition, $|\overline Y_{j,n}|\leq M+1$ on this event, so
$\widetilde B_n\leq M+3$.

By \eqref{eq:coarse-range}, clipping leaves the target unchanged on
$\mathcal G_n$. The final assertion of Lemma~\ref{lem:clipping} gives
\begin{equation}\label{eq:adaptive-loss-transfer}
\norm{C_{\widetilde B_n}\circ f-f_0}_{\Lcal^2(P_X)}^2
\leq K_\eta(M+3)D_\eta(f,f_0)
\quad\text{on }\mathcal G_n.
\end{equation}
Here $K_\eta(b)$ is increasing in $b>0$, as follows from the
monotonicity of $t/(1-e^{-t})$ for $t>0$.
Take $C^\prime=\max\{C,CK_\eta(M+3),3\}$ and $c^\prime=c$.
On $\mathcal G_n$, clipped loss exceeding $C^\prime\varepsilon_n^2$
implies $D_\eta(f,f_0)>C\varepsilon_n^2$.
The expected posterior probability is therefore bounded by the
R\'enyi tail plus $\PP_{f_0}(\mathcal G_n^c)$, hence by
$3e^{-cn\varepsilon_n^2}$.
Our choice $C^\prime\geq3$ proves \eqref{eq:l2-result}.
\end{proof}
\endgroup

If a fixed $B\geq M$ is prescribed, Lemma~\ref{lem:clipping}
gives $\norm{C_B\circ f-f_0}_{\Lcal^2(P_X)}^2
\leq K_\eta(B)D_\eta(f,f_0)$. The R\'enyi bound then gives a pure
exponential posterior tail at a suitable constant multiple of
$\varepsilon_n^2$. This comparison does not require estimating $B$.

\subsection{A clipping level for arbitrary designs}
\label{app:arbitrary-design}

\begingroup
For the alternative in Remark~\ref{rem:arbitrary-design}, set
$s_n=\max\{1,\lfloor n^{1/(4d)}\rfloor\}$.
Divide each coordinate interval into $s_n$ equal intervals and take
their Cartesian products. Use left-closed, right-open intervals,
except the last interval, which includes one. This gives a disjoint
partition of $\Xcal$ into $H_n=s_n^d\leq n^{1/4}$ cubes
$\mathcal Q_{j,n}^{\circ}$, each with max-norm diameter $1/s_n$.
Write $N_{j,n}^{\circ}$ and $\overline Y_{j,n}^{\circ}$ for their
counts and response means. Set $t_n=\lceil8\log(en)\rceil$.
A cell is eligible if it contains at least $t_n$ observations.
Define
\endgroup
\begin{equation}\label{eq:sparse-clipping}
\widetilde B_n^{\circ}=2+
\max_{j:N_{j,n}^{\circ}\geq t_n}|\overline Y_{j,n}^{\circ}|,
\qquad
\widetilde f_n^{\circ}(x)=
\int_{\Ccal(\Xcal)} C_{\widetilde B_n^{\circ}}(f(x))\,\Pi_n^\eta(\mathrm d f\mid\Dcal_n),
\end{equation}
with an empty maximum equal to zero. This procedure also uses neither
$\beta$ nor $M$.

\begingroup
\begin{proof}[Proof of Remark~\ref{rem:arbitrary-design}]
Choose the fixed R\'enyi threshold
$J_0>(1+\eta)/(1-\eta)$ as in the proof of Theorem~\ref{thm:main}.
Let $\mathcal G_n^{\circ}$ be the event that every eligible cell has
noise mean of absolute value at most one. Let
\[
\mathcal U_n^{\circ}
=\bigcup_{j:N_{j,n}^{\circ}<t_n}\mathcal Q_{j,n}^{\circ}
\]
be the union of the ineligible cells.

Conditional Gaussian tails and $N_{j,n}^{\circ}\geq t_n$ for
eligible cells give
\begin{equation}\label{eq:good-noise-bound}
\PP_{f_0}\{(\mathcal G_n^{\circ})^c\}
\leq2H_ne^{-t_n/2}\lesssim n^{-15/4}.
\end{equation}
For sufficiently large $n$, $Ms_n^{-\beta}\leq1$.
On $\mathcal G_n^{\circ}$, the clipping level lies in $[2,M+3]$.
Within every eligible cell, the averaging argument used in
\eqref{eq:coarse-range} gives
$|f_0(x)|\leq|\overline Y_{j,n}^{\circ}|+2
\leq\widetilde B_n^{\circ}$.
Thus clipping can change the target only on $\mathcal U_n^{\circ}$.
Since $|f_0|\leq M$, it follows that
\begin{equation}\label{eq:clipping-bias-mass}
\norm{C_{\widetilde B_n^{\circ}}\circ f_0-f_0}_{\Lcal^2(P_X)}^2
\leq M^2P_X(\mathcal U_n^{\circ})
\quad\text{on }\mathcal G_n^{\circ}.
\end{equation}

Write $p_j=P_X(\mathcal Q_{j,n}^{\circ})$.
For $p\in[0,1]$ and an integer $t\geq1$, if $N\sim\mathrm{Bin}(n,p)$ and
$N'\sim\mathrm{Bin}(n+1,p)$,
comparison of the binomial masses at $l$ and $l+1$ gives
\[
p\PP(N<t)=\frac1{n+1}\E\{N'\mathbf1_{\{N'\leq t\}}\}
\leq\frac{t}{n+1}.
\]
Applying this identity to each cell and summing yields, uniformly
over all $P_X$,
\begin{equation}\label{eq:underoccupied-mean}
\E_{f_0}P_X(\mathcal U_n^{\circ})
\leq\frac{H_nt_n}{n+1}
\lesssim n^{-3/4}\log(en)=o(\varepsilon_n^2).
\end{equation}
For the last comparison, $2\beta/(2\beta+d)\leq2/3$ under our
assumptions, whereas $3/4>2/3$.

We also need a tail bound for this random probability mass.
For $u>0$, if $P_X(\mathcal U_n^{\circ})>u$, the subset of ineligible cells has
total probability greater than $u$ and total count below $H_nt_n$.
For each fixed subset $I$, its total count is
$\mathrm{Bin}(n,p_I)$, where $p_I$ is its total probability.
If $p_I>u$ and $nu\geq2H_nt_n$, Chernoff's inequality bounds the
probability that its count is below $H_nt_n$ by $e^{-nu/8}$.
There are at most $2^{H_n}$ subsets, so
\begin{equation}\label{eq:underoccupied-tail}
\PP_{f_0}\{P_X(\mathcal U_n^{\circ})>u\}
\leq2^{H_n}e^{-nu/8}.
\end{equation}
Since $d/(2\beta+d)\geq1/3>1/4$, we have
$n\varepsilon_n^2\gg H_nt_n+H_n$.
Taking $u=a_0\varepsilon_n^2$ with any fixed $a_0>0$ therefore
makes \eqref{eq:underoccupied-tail} at most
$e^{-cn\varepsilon_n^2}$ for sufficiently large $n$.

On $\mathcal G_n^{\circ}\cap
\{P_X(\mathcal U_n^{\circ})\leq a_0\varepsilon_n^2\}$,
Lemma~\ref{lem:clipping} implies that
$D_\eta(f,f_0)\leq J_0\varepsilon_n^2$ entails
\[
\norm{C_{\widetilde B_n^{\circ}}\circ f-f_0}_{\Lcal^2(P_X)}^2
\leq\{2K_\eta(M+3)J_0+2M^2a_0\}\varepsilon_n^2.
\]
Lemma~\ref{lem:frac}, \eqref{eq:good-noise-bound}, and
\eqref{eq:underoccupied-tail} prove \eqref{eq:designfree-l2}.

For expected risk, conditional Jensen's inequality bounds the loss
of the posterior mean by the posterior mean of the clipped losses.
On $\mathcal G_n^{\circ}$, split that integral at
$D_\eta=J_0\varepsilon_n^2$.
The resulting bound is
\[
2K_\eta(M+3)J_0\varepsilon_n^2
+2M^2P_X(\mathcal U_n^{\circ})
+(2M+3)^2\Pi_n^\eta
\{D_\eta(f,f_0)>J_0\varepsilon_n^2\mid\Dcal_n\}.
\]
Its expectation is $O(\varepsilon_n^2)$ by
\eqref{eq:underoccupied-mean} and Lemma~\ref{lem:frac}.
On $(\mathcal G_n^{\circ})^c$,
$\widetilde B_n^{\circ}\leq2+M+\max_{i\leq n}|\xi_i|$,
so every clipped loss is at most
$(2M+2+\max_{i\leq n}|\xi_i|)^2$.
The fourth moment of this Gaussian maximum is $O(\log^2(en))$.
Cauchy--Schwarz and \eqref{eq:good-noise-bound} bound the exceptional
risk contribution by $O_M(n^{-15/8}\log(en))=o(\varepsilon_n^2)$.
This proves the upper risk rate in \eqref{eq:risk} for
$\widetilde f_n^{\circ}$ without any density assumption or moment assumption on the network parameters.
\end{proof}
\endgroup

\section{Minimax rates}\label{app:minimax}

\begingroup
\begin{proof}[Proof of Theorem~\ref{cor:risk}, upper bound]
Define the clipped-draw loss
$\ell_n(f)=\norm{C_{\widetilde B_n}\circ f-f_0}_{\Lcal^2(P_X)}^2$.
The posterior mean is an average of clipped draws, so conditional
Jensen's inequality gives
\[
\norm{\widetilde f_n-f_0}_{\Lcal^2(P_X)}^2
\leq\int_{\Ccal(\Xcal)}\ell_n(f)\,
\Pi_n^\eta(\mathrm d f\mid\Dcal_n).
\]
We bound this integral separately on the clipping event
$\mathcal G_n$ from the proof of Theorem~\ref{thm:main} and its
complement.

On $\mathcal G_n$, $\widetilde B_n\leq M+3$, so
$\ell_n(f)\leq(2M+3)^2$ for every $f$.
For draws satisfying $D_\eta(f,f_0)\leq J_0\varepsilon_n^2$,
\eqref{eq:adaptive-loss-transfer} gives the sharper bound
$K_\eta(M+3)J_0\varepsilon_n^2$.
Splitting the posterior integral into these two sets gives
\begin{equation}
\begin{split}
\mathbf1_{\mathcal G_n}\int_{\Ccal(\Xcal)}\ell_n(f)\,
\Pi_n^\eta(\mathrm d f\mid\Dcal_n)
&\leq K_\eta(M+3)J_0\varepsilon_n^2\\
&\quad +(2M+3)^2\Pi_n^\eta
\{D_\eta(f,f_0)>J_0\varepsilon_n^2\mid\Dcal_n\}.
\end{split}
\end{equation}
Lemma~\ref{lem:frac} bounds the expectation of the last posterior
probability by an exponentially small quantity. Hence the risk
contribution on $\mathcal G_n$ is $O(\varepsilon_n^2)$.

On $\mathcal G_n^c$, every cell response mean has absolute value
at most $M+\max_{i\leq n}|\xi_i|$. Thus
$\widetilde B_n\leq2+M+\max_{i\leq n}|\xi_i|$ and
$\ell_n(f)\leq(2M+2+\max_{i\leq n}|\xi_i|)^2$.
The fourth moment of $\max_{i\leq n}|\xi_i|$ is
$O(\log^2(en))$. Cauchy--Schwarz and
\eqref{eq:coarse-good-bound} therefore bound the exceptional
risk contribution by
\[
C_M\log(en)\sqrt{J_n}\exp\{-c_Xn/(16J_n)\}
=o(\varepsilon_n^2).
\]
Combining the two contributions proves \eqref{eq:risk}.
Only bounded clipped loss is integrated under the posterior;
no posterior moment of the R\'enyi divergence is needed.
\end{proof}
\endgroup

\begingroup
\begin{proof}[Proof of Theorem~\ref{cor:risk}, lower bound]
Choose a nonzero Lipschitz function
$\psi:[0,1]^d\to[0,1]$ supported in $[1/4,3/4]^d$,
and extend it by zero outside the unit cube. Let $C_\psi$ be a
Lipschitz constant of $\psi$ in maximum norm, let $m\geq1$ be an
integer, and let $a>0$ be an amplitude to be chosen below.
Partition $[0,1]^d$ into $m^d$ cubes $\mathcal Q_j$ of side
$h=1/m$, with lower corners $x_j$. Define
\[
\psi_j(x)=a h^\beta\psi((x-x_j)/h)
\mathbf1\{x\in\mathcal Q_j\},\qquad
f_\theta=\sum_{j=1}^{m^d}\theta_j\psi_j,
\qquad\theta\in\{0,1\}^{m^d}.
\]
Each bump has height at most $ah^\beta$ and Lipschitz constant
at most $C_\psi ah^{\beta-1}$.
Its difference between two values at maximum-norm distance $r$ is bounded by
$\min\{C_\psi ah^{\beta-1}r,2ah^\beta\}$, which is at most
$C a r^\beta$. Distinct bump supports are separated by at least
$h/2$ in max norm, and the bumps vanish near cell boundaries.
For points at distance less than $h/2$, at most one bump contributes
nonzero values, so its bound applies. For larger distances, the
height bound gives the same estimate with a fixed larger constant.
Thus the sums have H\"older seminorm at most $Ca$, and
$\norm{f_\theta}_\infty\leq a$.
Choose a fixed $a>0$ small enough that every
$f_\theta\in\Hcal_d^\beta(M)$, uniformly in $m$.

Indexing vectors at Hamming distance one differ in a single
coordinate $j$, so their targets differ by $\psi_j$ up to sign.
Put $I_\psi=\int_{[0,1]^d}\psi(u)^2\,\mathrm d u>0$.
The squared $\Lcal^2(P_X)$ distance between these targets, denoted by $\Delta_j^2$,
satisfies
\[
c_Xa^2h^{2\beta+d}I_\psi\leq\Delta_j^2
\leq C_Xa^2h^{2\beta+d}I_\psi.
\]
The $n$-sample Kullback--Leibler divergence between these two
Gaussian-regression experiments is $n\Delta_j^2/2$.
Take $m=\lceil n^{1/(2\beta+d)}\rceil$ and reduce the fixed $a$,
if necessary, so that $C_Xa^2I_\psi/2\leq1/8$.
Since $nh^{2\beta+d}\leq1$, the divergence between any such
neighbouring experiments is at most
$1/8$. Pinsker's inequality then bounds its total variation by $1/4$.

For any estimator $f_n$, choose $\widehat\theta_j\in\{0,1\}$ to
minimize
$\int_{\mathcal Q_j}(f_n(x)-b\psi_j(x))^2P_X(\mathrm d x)$
over $b\in\{0,1\}$, taking zero in a tie.
An estimator with infinite risk already satisfies the desired lower
bound, so it suffices to consider finite-risk estimators.
If the decision is wrong, its chosen candidate is no farther from
$f_n$ than the true candidate in the local norm
$(\int_{\mathcal Q_j}|g(x)|^2P_X(\mathrm d x))^{1/2}$.
The triangle inequality then bounds $\Delta_j$ by twice the local
estimation error, so the local squared loss is at least
$\Delta_j^2/4$.

Average over the uniform distribution on $\{0,1\}^{m^d}$.
For coordinate $j$, mixing over the other coordinates preserves the
$1/4$ total-variation bound by convexity. The average error
probability for each binary decision is therefore at least
$(1-1/4)/2=3/8$. Summing the local losses gives
\[
\inf_{f_n}\sup_{\theta}
\E_{f_\theta}\norm{f_n-f_\theta}_{\Lcal^2(P_X)}^2
\geq c\,m^da^2h^{2\beta+d}I_\psi
\asymp h^{2\beta}\asymp n^{-2\beta/(2\beta+d)}.
\]
The finite family lies in the stated H\"older ball, so the bound also
holds for minimax risk over that ball.
\end{proof}
\endgroup

\section{Bounded covariate supports}\label{app:bounded-support}

The contraction and risk upper bounds do not require the support
of $P_X$ to be a cube. Let $S\subset\mathbb R^d$ be a fixed bounded Borel set,
$P_X(S)=1$, and write $\norm f_{\infty,S}=\sup_{x\in S}|f(x)|$.
Define
\[
\Hcal^\beta(S;M)=\{f:S\to\mathbb R:
\norm f_{\infty,S}\leq M,\
|f(x)-f(y)|\leq M\norm{x-y}_\infty^\beta
\text{ for all }x,y\in S\}.
\]
The condition uses distances in the surrounding space $\mathbb R^d$,
including between different components of $S$. Choose a fixed closed cube $\mathcal Q\supset S$ of side length
$\ell>0$, and regard the {same network and parameter priors as a law}
on the continuous functions $\Ccal(\mathcal Q)$. The eligible-cell clipping rule uses this enclosing cube,
not the precise shape of $S$. Partition $\mathcal Q$ into $s_n^d$ cells
and use \eqref{eq:sparse-clipping} with their counts and means.

\begin{proposition}\label{prop:bounded-support}
For every fixed $0<\beta\leq1$ and $M>0$, the R\'enyi bound
\eqref{eq:renyi-result} and the clipped-draw and risk assertions of
Remark~\ref{rem:arbitrary-design} hold with $\Hcal^\beta(S;M)$
in place of $\Hcal_d^\beta(M)$ and $\Lcal^2(P_X)$ evaluated on $S$.
Their constants and the sample size beyond which they hold may
depend on $\mathcal Q$. Neither the pure
exponential adaptive bound \eqref{eq:l2-result} nor the regular-design
lower bound \eqref{eq:lower} is asserted for arbitrary $S$.
\end{proposition}

\begingroup
\begin{proof}
Lemma~\ref{lem:extension}, followed by clipping to $[-M,M]$,
extends each target from $S$ to $\mathbb R^d$ without increasing
its supremum bound or H\"older constant.
Choose a fixed
$R\geq\max\{1,\sup_{x\in\mathcal Q}\norm x_2\}$.
The rescaled extension $z\mapsto f_0(Rz)$ has a uniformly bounded
H\"older norm on the unit ball. Apply the argument of
Proposition~\ref{prop:approx} to this function and rescale the network
back to $\mathcal Q$.
The bounds on width and on the sum of absolute output coefficients retain their
orders in $\varepsilon$, and the approximating {hidden-layer coefficients and biases}
remain bounded by one.
In Lemma~\ref{lem:stability}, the bound for perturbing {a hidden-layer coefficient or bias changes only} by the fixed factor
$1+\sup_{x\in\mathcal Q}\norm x_1$.
The approximating centres are eventually admissible in
\eqref{eq:weight-centres}. The parameter-box condition
\eqref{eq:weight-mass} and the width condition \eqref{eq:width}
therefore give the same uniform prior-probability bound on $\mathcal Q$.
Lemma~\ref{lem:frac} then proves \eqref{eq:renyi-result}, with
constants allowed to depend on $\mathcal Q$.

Intersect the equal cells of the enclosing cube with $S$.
These intersections partition $S$ into at most $s_n^d$
sets, each of max-norm diameter at most $\ell/s_n$.
For sufficiently large $n$, $M(\ell/s_n)^\beta\leq1$.
The proofs of \eqref{eq:good-noise-bound}--\eqref{eq:underoccupied-tail}
use only the number of cells, their probabilities, and their
binomial counts. They therefore apply to these intersections,
including any having probability zero.
The clipped-loss and risk bounds follow as in
Remark~\ref{rem:arbitrary-design}.
The separate pure exponential bound for coarse-partition clipping
would require an additional lower bound on the probabilities of
its relevant cells; no such bound is assumed here.
\end{proof}
\endgroup

\section{Simulation design and computation}\label{app:simulation}

All simulation code and reproducibility materials are available at\\
\mbox{\url{https://github.com/hiendn/bayesian-relu-minimax}}.
The study uses likelihood power $\eta=0.9$ and sample sizes from
64 to 4096. For each target, sample size and parameter family, we fit
50 replications. The Gaussian width-prior comparisons use the same
fits within each target and sample size. Supplementary Tables~\ref{tab:supp-half}
and~\ref{tab:supp-lipschitz} give all mean errors, Monte Carlo standard
errors and pointwise 95\% Monte Carlo confidence intervals.

\subsection{Targets, prior, and loss}
The numerical study uses $d=1$, $P_X=\Unif[0,1]$, and independent
$\mathrm N(0,1)$ errors. {The two fixed targets are
\begin{equation}\label{eq:simulation-truth}
\begin{aligned}
f_{1/2}(x)&=\sqrt{\abs{\sin(2\pi x)}},\\
f_1(x)&=\abs{\sin(2\pi x)},
\end{aligned}
\qquad x\in[0,1].
\end{equation}
Both have supremum norm one. For every $x,y\in[0,1]$,
$\bigl|\abs{\sin(2\pi x)}-\abs{\sin(2\pi y)}\bigr|
\leq2\pi|x-y|$ and
$|\sqrt u-\sqrt v|\leq\sqrt{|u-v|}$ for $u,v\geq0$ give
\[
\begin{aligned}
|f_{1/2}(x)-f_{1/2}(y)|&\leq\sqrt{2\pi}\,|x-y|^{1/2},\\
|f_1(x)-f_1(y)|&\leq2\pi|x-y|.
\end{aligned}
\]
Thus $f_{1/2}\in\Hcal_1^{1/2}(\sqrt{2\pi})$ and
$f_1\in\Hcal_1^1(2\pi)$; in particular, both belong to their
respective classes with the common radius $2\pi$.
Both functions have two peaks, at $1/4$ and $3/4$.
At the interior zero $x_0=1/2$,
\[
f_{1/2}(x_0+h)-f_{1/2}(x_0)
=\sqrt{2\pi}\,|h|^{1/2}+O(|h|^{5/2}),
\]
so $f_{1/2}$ has no larger H\"older exponent at this point and is
not Lipschitz. The one-sided derivatives of $f_1$ differ by $4\pi$
at $x_0$. Within each lobe,
$f_1''=-(2\pi)^2f_1<0$, so $f_1$ is not piecewise linear.
Every finite shallow ReLU network is Lipschitz and piecewise linear.
Consequently, neither target is represented exactly by such a
network. Both functions are evaluated directly using elementary
operations.} The clipping procedure uses neither smoothness
parameter nor radius.

The Gaussian baseline has $\pi_{\mathsf K}(k)=2^{-k}$,
$U_m,B_m\sim\mathrm N(0,1)$, and
$A_m\mid\mathsf K=k\sim\mathrm N(0,k^3)$, {independently across all scalar network parameters.} The fractional power is $\eta=0.9$. The same fixed prior
is used for both targets and every sample size. Numerical computation
conditions on $\mathsf K\leq64$; the theorem uses unbounded width.

{The additional uniform and Student $t_{10}$ comparisons use
geometric width. In the uniform case, $U_m,B_m\sim\Unif[-2,2]$ and
$A_m\mid\mathsf K=k\sim\Unif[-k^2,k^2]$. In the Student case,
all network parameters independently follow a unit-scale $t_{10}$ law,
with no width-dependent scaling. Their Figure~\ref{fig:adaptive-risk}
labels are U/G and $t_{10}$/G.}

{Original replicate indices $1,\ldots,50$ identify the data streams. For a fixed target, the same indexed stream is used across parameter families and sample sizes, with nested data as $n$ increases. Every original replication is included, and neither risk nor credible coverage is used to select a computation stage. The released seeds and settings reproduce these fits.}

\subsection{Adaptive clipping and prediction risk}
\label{app:simulation-adaptive-risk}
For each data set we divide $[0,1]$ into $m_n=\lceil\log n\rceil$
equal cells and calculate the cutoff in \eqref{eq:adaptive-clipping}, setting an empty
cell mean to zero.
The same responses determine the clipping and the fractional posterior.
We average draws clipped at $\widetilde B_n$ as in
\eqref{eq:adaptive-mean}. Since $\widetilde B_n\geq2$ and both
{targets are bounded by one, the clipping interval always covers each
target in this study.}

{Across the reported fits, the smallest coarse-cell counts were 4 for $\beta=1/2$ and 6 for $\beta=1$.}

Write
$\Pi_{n,64}^{0.9}$ for the numerical
fractional posterior conditional on $\mathsf K\leq64$ and
$q_{n,\beta}=(\log n/n)^{2\beta/(2\beta+1)}$.

We approximate integrals under the uniform design by equally weighted
averages over a fixed grid, with one point in each of 1009 equal
subintervals:
\begin{equation}\label{eq:simulation-grid}
x_g=\frac{g-1/2+\sqrt2-1}{1009},\qquad g=1,\ldots,1009.
\end{equation}
This shifted rectangle rule uses a fixed offset of $\sqrt2-1$ grid
spacings so that its points do not coincide with the targets' central
zero at $1/2$ or peaks at $1/4$ and $3/4$.
Related uses of irrational evaluation points to avoid exact alignment
are discussed by \citet{bartuska2025} and \citet{shampine2008}.
These precedents motivate avoiding exact alignment; they do not
establish improved accuracy or a convergence rate for our chosen offset.
The same grid is used for every target, prior, sample size
and replication; it is fixed independently of the observations and
fitted functions. For the Gaussian fits in
Figure~\ref{fig:both-credible}, we independently recomputed the clipped
posterior means from the sampled widths and hidden-layer parameters,
integrating over the output coefficients, and then recomputed their
squared errors on this grid. These were compared with the estimates
produced during fitting.
The largest absolute discrepancies were below $10^{-10}$.
This agreement checks the calculation of the estimates; it does not
bound numerical integration error or residual Markov-chain error.

If $\widetilde f_{n,r}$ is the numerical adaptive posterior mean for
replicate $r$, its approximate integrated squared error is
\begin{equation}\label{eq:simulation-risk}
\widetilde R_{n,r,\beta}=\frac1{1009}\sum_{g=1}^{1009}
\left\{\widetilde f_{n,r}(x_g)-f_\beta(x_g)\right\}^2.
\end{equation}
{For Gaussian network parameters, we integrate the clipped prediction conditional
on the width, hidden-layer coefficients and biases, then average over
the posterior samples as described in Appendix~\ref{app:simulation-computation}.
This integration reduces Monte Carlo noise. Clipping the posterior
mean after averaging would define a different estimator.}

Each reported risk is the mean of 50 original diagnostic-passing
replicate errors. Let $\overline R_{n,\beta}$ be their mean and
$s_{n,\beta}$ their sample standard deviation. The Monte Carlo
standard error is $s_{n,\beta}/\sqrt{50}$. We report pointwise 95\%
asymptotic Monte Carlo confidence intervals using the usual
normal approximation. Let $\Phi$ be the standard normal distribution
function and $z_{0.975}=\Phi^{-1}(0.975)\approx1.95996$. The interval is
\[
\overline R_{n,\beta}\;\pm\;
z_{0.975}\frac{s_{n,\beta}}{\sqrt{50}}.
\]
These intervals quantify simulation uncertainty in the mean squared error
across all 50 original data replications within each case. They are not posterior
credible intervals and do not bound residual Markov-chain bias.
Since each target is bounded by one, the clipping bound satisfies
$\widetilde B_n\leq3+\max_{i\leq n}|\xi_i|$, where $\xi_i$ are the regression errors in \eqref{eq:model}, and each replicate
error is at most $(4+\max_{i\leq n}|\xi_i|)^2$. It therefore has
finite variance, supporting the normal approximation as the number
of reported completed replications increases.
In each risk-figure panel, the upper-rate scale $q_{n,\beta}$ and
lower-rate scale $n^{-2\beta/(2\beta+1)}$ are both normalized to the
Gaussian/geometric mean error at $n=64$. Their vertical constants
are illustrative and do not estimate the unknown theoretical
constants. The minimax lower rate is a class benchmark rather than
a lower bound on the error at either fixed target.

For adjacent sample sizes and for the two endpoints, we also report
the decay exponent
\[
p_{a,b}=\frac{\log(\overline R_{a,\beta}/\overline R_{b,\beta})}
{\log(b/a)},\qquad a<b.
\]
Its 95\% percentile bootstrap interval uses 2000 resamples of the 50 original replicate indices, paired across sample sizes for each target. This is a nonparametric percentile bootstrap applied to paired replication errors \citep[Sections~2, 7--8]{efron1985bootstrap}: the two errors from a replication form one resampling unit, and the interval endpoints are the empirical 2.5th and 97.5th percentiles of the recalculated exponents. This preserves dependence from the reused data streams. These fixed-target intervals do not bound residual Markov-chain error.

These exponents describe the displayed targets over a finite range of
sample sizes and do not estimate worst-case H\"older-ball rates.

{The function estimates in Figure~\ref{fig:lipschitz-credible} use original replicate one for both targets at $n=1024$ and $4096$, under uniform parameters and geometric width. We evaluate and clip the regression function for each posterior draw of
the output coefficients and hidden parameters, combine the draws from
the independent ladders, and calculate the posterior mean and empirical
type-1 equal-tailed 95\% intervals. Repeating these simulations with the
same data, seeds and settings reproduces the predictions and diagnostic
summaries exactly. The intervals concern the latent
regression function, with no observation noise added. The true
central-notch value zero lies in none of the intervals. At $n=4096$, the H\"older $1/2$ mean is $0.400525$ with interval $[0.194452,0.593273]$; the Lipschitz mean is $0.168787$ with interval $[0.025630,0.298567]$.

Figure~\ref{fig:both-credible} shows original replicate one for both targets at $n=1024$ and $4096$, under Gaussian parameters and geometric width. Conditional on each posterior draw of width and hidden parameters,
the regression function is normal with the mean and variance below. We invert the equal-weight mixture at probabilities $0.025$ and $0.975$, then clip the quantiles at $\pm\widetilde B_n$. The intervals include output-coefficient, width and hidden-parameter uncertainty. At the central zero, none of the intervals contain zero. At $n=4096$, the H\"older $1/2$ mean is $0.373204$ with interval $[0.186481,0.570900]$; the Lipschitz mean is $0.155984$ with interval $[0.003281,0.280728]$.}

\begin{figure}[ht]
\centering
\includegraphics[width=\textwidth]{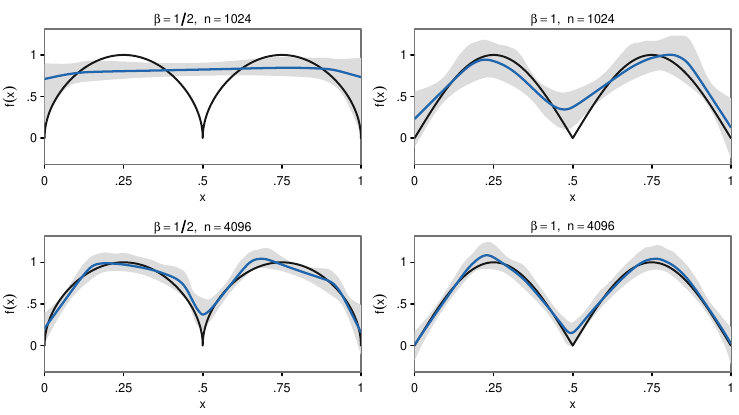}
\caption{{True functions, clipped posterior means, and
pointwise 95\% credible intervals for both targets at $n=1024,4096$.
Original replicate one; Gaussian parameters, geometric width, $\eta=0.9$,
and $\mathsf K\leq64$.}}
\label{fig:both-credible}
\end{figure}

\subsection{Posterior computation}\label{app:simulation-computation}

\begingroup
\paragraph{Numerical target.}
The accompanying R code implements the sampling transitions directly.
For this one-dimensional study, a full network state is
$\theta=(k,a,u,b)$, with $a,u,b\in\mathbb R^k$, on
$\mathcal T_{64}=\bigsqcup_{k=1}^{64}
(\{k\}\times\mathbb R^{3k})$.
Write $p_k(a,u,b)$ for the conditional parameter density of the
chosen row of Table~\ref{tab:families}, and
$\pi_{64}(k)=2^{-k}/(1-2^{-64})$, $1\leq k\leq64$, for the
conditioned geometric width law. At likelihood power
$0\leq\lambda\leq\eta=0.9$, the target density is
\begin{equation}\label{eq:simulation-target}
p_{\lambda,64}(\theta\mid\Dcal_n)\propto
\pi_{64}(k)p_k(a,u,b)
\exp\{-\lambda R(\theta)/2\},\qquad
R(\theta)=\sum_{i=1}^n\{Y_i-f_\theta(X_i)\}^2.
\end{equation}
The reference measure counts widths and uses Lebesgue measure on
the scalar network parameters. The likelihood uses the unclipped network;
clipping enters the prediction average in \eqref{eq:adaptive-mean}.
The chain at power $0.9$ targets the numerical fractional posterior.
At power zero the target is the conditioned prior.

\endgroup

\paragraph{{Gaussian {output coefficients}.}}
For Gaussian network parameters, the {output coefficients} can be integrated
analytically. Conditional on width $k$ and the {hidden-layer coefficients and biases},
write $\boldsymbol\Phi$ for the $n\times k$ matrix with entries
$\rho(U_mX_i+B_m)$ and $Y=(Y_1,\ldots,Y_n)^\top$.
Let $I_j$ denote the $j\times j$ identity matrix. For vector arguments,
$\mathrm N(\mu,\Sigma)$ denotes the normal distribution with mean
vector $\mu$ and covariance matrix $\Sigma$. In this
one-dimensional experiment, the output vector has prior
$\mathrm N(0,k^3I_k)$.
At likelihood power $0\leq\lambda\leq\eta$, define
\begin{equation}\label{eq:simulation-collapse}
Q_\lambda=k^{-3}I_k+\lambda\boldsymbol\Phi^\top\boldsymbol\Phi,
\qquad
m_\lambda=\lambda Q_\lambda^{-1}\boldsymbol\Phi^\top Y.
\end{equation}
For an {output coefficient vector} $a\in\mathbb R^k$, multiplying
the prior density by the powered Gaussian likelihood and
completing the square gives
\[
k^{-3}a^\top a+\lambda\norm{Y-\boldsymbol\Phi a}_2^2
=(a-m_\lambda)^\top Q_\lambda(a-m_\lambda)
+\lambda Y^\top Y-m_\lambda^\top Q_\lambda m_\lambda.
\]
Hence the conditional output posterior is
$\mathrm N(m_\lambda,Q_\lambda^{-1})$.
If $\varphi_k(a;0,k^3I_k)$ is the Gaussian prior density, integration
of this completed square yields
\[
\begin{split}
&\int_{\mathbb R^k}(2\pi)^{-\lambda n/2}
 e^{-\lambda\norm{Y-\boldsymbol\Phi a}_2^2/2}
 \varphi_k(a;0,k^3I_k)\,\mathrm d a\\
&\qquad=(2\pi)^{-\lambda n/2}k^{-3k/2}\det(Q_\lambda)^{-1/2}
 \exp\left\{-\frac\lambda2Y^\top Y
 +\frac12m_\lambda^\top Q_\lambda m_\lambda\right\}.
\end{split}
\]
Using
$k^{3k}\det(Q_\lambda)
=\det(I_k+\lambda k^3\boldsymbol\Phi^\top\boldsymbol\Phi)
=\det(I_n+\lambda k^3\boldsymbol\Phi\boldsymbol\Phi^\top)$
and the Woodbury identity yields
\begin{equation}\label{eq:simulation-marginal}
\det(I_n+\lambda k^3\boldsymbol\Phi\boldsymbol\Phi^\top)^{-1/2}
\exp\left\{-\frac\lambda2Y^\top
(I_n+\lambda k^3\boldsymbol\Phi\boldsymbol\Phi^\top)^{-1}Y\right\},
\end{equation}
after omitting $(2\pi)^{-\lambda n/2}$, which is independent of the
network state.
We evaluate the factor using a $k\times k$ Cholesky factorization of
$Q_\lambda$. The implementation also omits $-\lambda Y^\top Y/2$
from its log value. These terms depend only on the data and the assigned power, so they
cancel in both within-level moves and exchanges between levels;
all terms depending on width or
{hidden-layer coefficients and biases} are retained. The calculation takes
$O(nk^2+k^3)$ time and $O(nk+k^2)$ working memory.
Grid prediction adds $O(1009k^2)$ time.

\begingroup
For a Gaussian state $s=(k,u,b)$, containing the width,
hidden-layer coefficients and biases, let
$\ell_\lambda(s)$ be the log of the factor in
\eqref{eq:simulation-marginal}, and let $\varphi$ be the standard
normal density. The sampled marginal density is
\[
p^{\mathrm G}_{\lambda,64}(s\mid\Dcal_n)\propto
\pi_{64}(k)\prod_{m=1}^k\varphi(u_m)\varphi(b_m)
e^{\ell_\lambda(s)}.
\]
The prior densities of the hidden-layer coefficients and biases remain in this density. Gaussian
integration removes the $k$ {output coefficients} from the sampled
state and retains their exact conditional distribution.
The factor is evaluated separately at each power: integrating a
powered likelihood generally differs from powering its integral.

\paragraph{Gaussian Metropolis--Hastings updates.}
At a fixed width, for a proposal density $q_\lambda(s'\mid s)$,
the acceptance rule is
\[
1\wedge
\frac{p^{\mathrm G}_{\lambda,64}(s'\mid\Dcal_n)
q_\lambda(s\mid s')}
{p^{\mathrm G}_{\lambda,64}(s\mid\Dcal_n)
q_\lambda(s'\mid s)}.
\]
Each iteration selects a width change with probability $0.25$.
Otherwise it chooses one unit uniformly. With probability $0.2$,
the hidden-layer coefficient and bias of that unit are independently redrawn
from $\mathrm N(0,1)$. The prior and proposal densities cancel,
giving acceptance probability $1\wedge e^{\Delta\ell_\lambda}$,
where $\Delta\ell_\lambda=\ell_\lambda(s')-\ell_\lambda(s)$.
Otherwise either its hidden-layer coefficient or its bias is chosen
with equal probability. Given its current value $z$, the proposed value
is the random variable $Z'=z+\zeta$, where
$\zeta\sim\mathrm N(0,0.65^2)$. For a realized proposal $z'$, the
acceptance probability is
\[
1\wedge\exp\{\Delta\ell_\lambda-(z'^2-z^2)/2\}.
\]
The proposal is symmetric, and the final term is the Gaussian
prior-density ratio for the selected parameter.

For width changes, define the birth probability
\[
b_k=\begin{cases}
1,&k=1,\\
1/2,&1<k<64,\\
0,&k=64,
\end{cases}
\qquad d_k=1-b_k.
\]
A birth chooses one of the $k+1$ insertion positions uniformly
and draws the new unit's hidden-layer coefficient and bias from their Gaussian priors.
A death chooses one of the $k$ active positions uniformly and
deletes that unit. Their acceptance probabilities are
\begin{equation}\label{eq:simulation-gaussian-dimension}
\alpha_{\mathrm b}=1\wedge
\left(e^{\Delta\ell_\lambda}\frac{d_{k+1}}{2b_k}\right),
\qquad
\alpha_{\mathrm d}=1\wedge
\left(e^{\Delta\ell_\lambda}\frac{2b_{k-1}}{d_k}\right),
\end{equation}
for permitted births and deaths, respectively.
Forward insertion and reverse deletion choose the same position
with equal probability, and the prior density of the new hidden-layer coefficient and bias cancels
with its proposal density. The factors $1/2$ and $2$ are width-prior
ratios. The insertion map has Jacobian magnitude one; units are
represented in an ordered list, so no factorial correction is
needed \citep{green1995}. The marginal factor is recomputed at
the proposed width, including the changed $k^3$ output-prior
variance and all width-dependent normalizers.

\paragraph{Uniform and Student $t_{10}$ updates.}
The uniform and Student $t_{10}$ samplers retain output coefficients,
hidden-layer coefficients and biases, and use the full density in
\eqref{eq:simulation-target}. Let $z_m=(a_m,u_m,b_m)$ and
$g_k(z_m)$ be the product of the prior densities of these three parameters,
so that $p_k=\prod_{m=1}^k g_k(z_m)$.
A birth draws the new triple from $g_{k+1}$ and inserts it at a
uniformly chosen position; a death deletes a uniformly chosen
triple. For a birth from $\theta$ to $\theta'$, the acceptance rule is
\begin{equation}\label{eq:simulation-direct-birth}
1\wedge\left[
e^{-\lambda\{R(\theta')-R(\theta)\}/2}
\frac{d_{k+1}}{2b_k}
\prod_{m=1}^k\frac{g_{k+1}(z_m)}{g_k(z_m)}
\right].
\end{equation}
For a death from width $k$ to $k-1$, write
$z'_1,\ldots,z'_{k-1}$ for the remaining triples. Its rule is
\[
1\wedge\left[
e^{-\lambda\{R(\theta')-R(\theta)\}/2}
\frac{2b_{k-1}}{d_k}
\prod_{m=1}^{k-1}\frac{g_{k-1}(z'_m)}{g_k(z'_m)}
\right].
\]
The inserted or deleted triple's density cancels with its proposal
factor. Prior densities for the retained units must still be evaluated at
the destination width. In the uniform case the birth product is
$(k/(k+1))^{2k}$; the death product is
$(k/(k-1))^{2(k-1)}$ when every remaining
$\abs{a'_m}\leq(k-1)^2$, and is zero otherwise. For Student network parameters,
the products equal one because the parameter priors do not depend
on width.

Each direct-sampler iteration makes one width move with probability
$0.25$, otherwise one update of a hidden-layer coefficient or bias,
followed by updates of all output coefficients in a random order.
For the update of a hidden-layer coefficient or bias, the sampler chooses
one unit uniformly and then chooses its coefficient or bias with equal
probability. With probability $0.25$ it redraws the selected parameter from its prior,
giving acceptance probability
$1\wedge e^{-\lambda\{R(\theta')-R(\theta)\}/2}$.
Otherwise it proposes a Gaussian random walk with standard
deviation $0.4$, retaining the selected parameter's prior-density
ratio in the Metropolis--Hastings rule. When the parameter prior is uniform, proposals outside its
support are rejected.

For the updates of the output coefficients, put
$\phi_m=(\rho(u_mX_i+b_m))_{i=1}^n$,
$r_{-m}=Y-\sum_{\substack{j=1\\j\ne m}}^k a_j\phi_j$,
$v_m=\norm{\phi_m}_2^2$, and $c_m=\phi_m^\top r_{-m}$.
For uniform output coefficients, if $\lambda v_m>0$ the conditional law is
\[
a_m\mid\text{other parameters},\Dcal_n
\sim\mathrm N\!\left(c_m/v_m,(\lambda v_m)^{-1}\right)
\quad\text{truncated to }[-k^2,k^2].
\]
If $\lambda v_m=0$, it is the uniform prior on that interval.
For $\alpha,\gamma>0$, $\operatorname{IG}(\alpha,\gamma)$ denotes the inverse gamma distribution with density proportional to
$v^{-\alpha-1}e^{-\gamma/v}$ on $v>0$.
For Student output coefficients, introduce an auxiliary variance
$V_m\sim\operatorname{IG}(5,5)$ with
$a_m\mid V_m\sim\mathrm N(0,V_m)$.
The two conditional updates are
\begin{equation}\label{eq:simulation-student-gibbs}
\begin{split}
V_m\mid a_m&\sim
\operatorname{IG}\!\left(11/2,(10+a_m^2)/2\right),\\
a_m\mid V_m,\text{other parameters},\Dcal_n&\sim
\mathrm N\!\left(
\frac{\lambda c_m}{\lambda v_m+V_m^{-1}},
(\lambda v_m+V_m^{-1})^{-1}\right).
\end{split}
\end{equation}
They define a Gibbs transition for the {output coefficient}, with
the auxiliary variance regenerated during each sweep.
Residuals are updated after each output coefficient. These conditional
updates and the preceding Metropolis--Hastings moves preserve
the target at their assigned power.
At width $k$, network evaluation and the vector arithmetic in a
sweep over the output coefficients take $O(nk)$ operations per tempering level and
$O(nk)$ working memory, excluding retained diagnostics. A grid
prediction takes $O(1009k)$ operations.

\paragraph{Exchanges between likelihood powers.}
Parallel tempering \citep{jasra2007} uses 16 interacting chains at fixed powers
\[
\begin{gathered}
0,0.0009,0.0018,0.0036,0.0072,0.0144,0.0288,0.0576,
\\
0.1152,0.2304,0.3456,0.4608,0.576,0.6912,0.8064,0.9.
\end{gathered}
\]
This set of chains is called a ladder. Its invariant joint
distribution is the product of the 16 distributions at these powers. Every iteration updates every level once
as described above. Every fifth iteration proposes exchanges
between disjoint adjacent pairs, alternating
$(1,2),(3,4),\ldots$ and $(2,3),(4,5),\ldots$.
For {Gaussian states} $s,t$ at powers $\lambda,\lambda'$, the
exchange probability is
\[
1\wedge\exp\{\ell_\lambda(t)+\ell_{\lambda'}(s)
-\ell_\lambda(s)-\ell_{\lambda'}(t)\}.
\]
Each state is evaluated at its destination power and the common
prior terms cancel. For the direct samplers it simplifies to
\[
1\wedge\exp\left\{\frac{\lambda-\lambda'}2
\bigl[R(\theta)-R(\theta')\bigr]\right\},
\]
where $\theta,\theta'$ are the states currently at
$\lambda,\lambda'$. The exchanges preserve the joint target.
Only states at the final power $0.9$ are retained for estimation.

\paragraph{Prediction from posterior samples.}
For a {Gaussian state} at power $\eta$, let
$\boldsymbol\phi_x=(\rho(u_mx+b_m))_{m=1}^k$.
The conditional latent prediction is normal with
$\mu_x=\boldsymbol\phi_x^\top m_\eta$ and
$\sigma_x^2=\boldsymbol\phi_x^\top Q_\eta^{-1}\boldsymbol\phi_x$.
Put $B=\widetilde B_n$ for this data set.
If $\sigma_x>0$, set
$\alpha_x=(-B-\mu_x)/\sigma_x$ and
$\gamma_x=(B-\mu_x)/\sigma_x$. The conditional clipped mean is
\begin{equation}\label{eq:simulation-clipped-normal}
\begin{split}
h_x(s)={}&-B\Phi(\alpha_x)+B\{1-\Phi(\gamma_x)\}\\
&+\mu_x\{\Phi(\gamma_x)-\Phi(\alpha_x)\}
+\sigma_x\{\varphi(\alpha_x)-\varphi(\gamma_x)\}.
\end{split}
\end{equation}
If $\sigma_x=0$, set $h_x(s)=C_B(\mu_x)$.
For the combined posterior samples $s_1,\ldots,s_S$, the Gaussian estimate
is $S^{-1}\sum_{j=1}^S h_x(s_j)$. The direct samplers instead
average $C_B(f_{\theta_j}(x))$ over draws of all network parameters
$\theta_1,\ldots,\theta_S$.
Both averages approximate the clipped posterior mean in
\eqref{eq:adaptive-mean}; $\sigma_x^2$ is the conditional variance
of the latent function, without observation-noise variance.

\paragraph{Initialization, retention, and diagnostics.}
Separate ladders use independent seeds. For Gaussian network parameters the
chain at power $\eta=0.9$ in ladder $j$ starts at width $j$, giving widths
$1,\ldots,L$ for a run with $L$ ladders.
Other Gaussian levels, and all direct-sampler levels, start at
$k=\min\{64,1+Z\}$, where
$\PP(Z=z)=2^{-(z+1)}$, $z=0,1,\ldots$.
Network parameters are drawn from their conditional priors;
Gaussian {output coefficients} are integrated rather than initialized.
These are starting distributions; the acceptance ratios use the
conditioned prior in \eqref{eq:simulation-target}.
\endgroup

{Table~\ref{tab:simulation-computation-stages} specifies the eight computation stages, all fixed before this $\eta=0.9$ study. A later stage is considered only after all preceding stages fail the unchanged diagnostic screens, and the first accepted stage is retained. Every stage uses the same data, priors, 16 likelihood powers, parameter updates, clipping, width cap and diagnostic thresholds. Independent chain seeds depend on the stage index. We discard the first half of each ladder as burn-in and retain every twentieth subsequent state at power $\eta=0.9$.

\begin{table}[ht]
\centering\small
\caption{Computation stages. Later stages are used only after every
preceding stage fails the unchanged diagnostics. All eight stages were fixed before the study.}
\label{tab:simulation-computation-stages}
\begin{tabular}{@{}rrrr@{}}
\toprule
Stage & Independent ladders & Iterations per ladder & Posterior draws\\
\midrule
1 & 3 & 24,000 & 1,800\\
2 & 4 & 48,000 & 4,800\\
3 & 6 & 96,000 & 14,400\\
4 & 8 & 192,000 & 38,400\\
5 & 8 & 384,000 & 76,800\\
6 & 8 & 768,000 & 153,600\\
7 & 8 & 1,536,000 & 307,200\\
8 & 8 & 3,072,000 & 614,400\\
\bottomrule
\end{tabular}
\end{table}

We require potential scale reduction factor
$\widehat R\leq1.10$ \citep{gelman1992inference} and pooled effective sample size at least 50
for width and for clipped predictions at five points near
$0.1,0.3,0.5,0.7,0.9$.
For uniform and Student network parameters, the corresponding unclipped
prediction diagnostics must also pass.
We calculate these diagnostics with \texttt{coda}
\citep{plummer2006coda}.

For each diagnostic, the pooled effective sample size is the sum of
the per-ladder estimates. We also require the width to remain below
64 at every post-burn-in iteration at likelihood power $0.9$,
including iterations omitted by thinning. For each parameter family,
target and sample size, the fits for all 50 replications pass these
checks. Estimated risk is never used to select fits or computation stages.

The calculations used R~4.4.0 on an arm64 macOS computer. The repository
provides the settings and seeds for each reported fit. For the
displayed Gaussian fits, independently recomputing posterior means and
squared errors from the sampled parameters gives agreement to within
$10^{-10}$, as detailed in Appendix~\ref{app:simulation-adaptive-risk}.}

{All fits pass the diagnostic criteria. The original eight computation stages were specified before the study. After the sole unresolved original fit, U/G for the H\"older $1/2$ target at $n=4096$, replication 15, failed all eight stages, four additional stages were specified before their execution. These stages use 6.144, 12.288, 24.576 and 49.152 million iterations per ladder, respectively, with eight independent ladders, half burn-in and thinning 20; the first passing stage is retained. The unresolved fit first passed the diagnostics at stage 10. The original data, seed formula, priors, sampler, 16 likelihood-power levels and diagnostic thresholds are unchanged. No fit occupies the width cap after burn-in. These diagnostic screens do not prove exploration of every posterior mode or bound residual sampling error.}

{The finite chain lengths and the use of grid averages to approximate
design integrals may introduce residual error.
Conditioning on width at most 64 changes the prior, and convergence
diagnostics from the sampled chains do not bound posterior mass
beyond that limit.}

{Under geometric width and Gaussian network parameters, mean errors at $n=64$ and $4096$ are $0.08501$ and $0.00597$ for $\beta=1/2$, and $0.12262$ and $0.00408$ for $\beta=1$. The endpoint exponents $p_{64,4096}$ are $0.639$ and $0.818$, with 95\% paired-bootstrap intervals $[0.609,0.670]$ and $[0.791,0.846]$. These finite-range comparisons describe fixed targets rather than worst-case H\"older-ball risk.}

Table~\ref{tab:simulation-adaptive-risk} gives the Gaussian/geometric
errors and clipping means at all sample sizes.
\begin{table}[t]
\caption{Squared prediction error with Gaussian parameters and geometric width. Each estimate uses 50 independent replications. The SE column gives Monte Carlo standard errors. Supplementary Tables~\ref{tab:supp-half} and~\ref{tab:supp-lipschitz} give pointwise 95\% normal Monte Carlo confidence intervals.}
\label{tab:simulation-adaptive-risk}
\centering\small
\begin{tabular}{@{}crcccc@{}}
\toprule
$\beta$ & $n$ & Mean $\widetilde B_n$ & Error & SE & Error/$q_{n,\beta}$\\
\midrule
$1/2$ & 64 & 3.22 & 0.0850 & 0.0038 & 0.333\\
 & 128 & 3.10 & 0.0745 & 0.0035 & 0.383\\
 & 256 & 3.08 & 0.0633 & 0.0014 & 0.430\\
 & 512 & 3.06 & 0.0508 & 0.0018 & 0.460\\
 & 1024 & 3.02 & 0.0228 & 0.0020 & 0.278\\
 & 2048 & 3.02 & 0.0100 & 0.0004 & 0.164\\
 & 4096 & 3.02 & 0.0060 & 0.0002 & 0.132\\
\midrule
$1$ & 64 & 3.09 & 0.1226 & 0.0040 & 0.759\\
 & 128 & 3.02 & 0.1062 & 0.0027 & 0.941\\
 & 256 & 3.05 & 0.0911 & 0.0027 & 1.173\\
 & 512 & 3.03 & 0.0407 & 0.0046 & 0.769\\
 & 1024 & 2.99 & 0.0113 & 0.0007 & 0.317\\
 & 2048 & 2.96 & 0.0067 & 0.0003 & 0.277\\
 & 4096 & 2.99 & 0.0041 & 0.0002 & 0.255\\
\bottomrule
\end{tabular}
\end{table}

\subsection{Sensitivity to the fixed width law}
\label{app:simulation-width}
We keep the Gaussian conditional network-parameter law and use the fixed width
laws
\begin{equation}\label{eq:simulation-width-laws}
\pi_G(k)=2^{-k},\qquad
\pi_P(k)=\frac{e^{-1}}{(k-1)!},\qquad
\pi_H(k)=\frac{k^{-5/2}}{\sum_{j=1}^\infty j^{-5/2}},\quad k\geq1.
\end{equation}
They satisfy \eqref{eq:width} and have prior mean widths $2,2$, and
approximately $1.95$, respectively. Each numerical law is
conditioned on $\mathsf K\leq64$.

Let $Q_{n,64}^{\eta,\pi}$ be the posterior on $\mathcal T_{64}$ of
$(\mathsf K,\theta_{\mathsf K})$ under width law $\pi$,
conditional on $\mathsf K\leq64$.
For any integrable function $h$ of the pair,
\begin{equation}\label{eq:simulation-width-reweight}
\int_{\mathcal T_{64}} h\,\mathrm d Q_{n,64}^{\eta,\pi}
=\frac{\int_{\mathcal T_{64}} h\,w_\pi(\mathsf K)
\,\mathrm d Q_{n,64}^{\eta,\pi_G}}
{\int_{\mathcal T_{64}} w_\pi(\mathsf K)
\,\mathrm d Q_{n,64}^{\eta,\pi_G}},\qquad
w_\pi(k)=\frac{\pi(k)}{\pi_G(k)}.
\end{equation}
The conditional network-parameter law is unchanged and the truncation
normalizers cancel. We apply this identity to the conditional
expectation of the clipped output, using the same clipping and data
for all laws. Importance weights are normalized within each data set
before calculating its posterior mean and squared error.
For the $S$ posterior draws with widths $k_1,\ldots,k_S$, set
$w_s=w_\pi(k_s)$. The Kish effective sample size, expressed as a
fraction of $S$ \citep[pp.~190--191]{kish1992},
$(\sum_{s=1}^{S}w_s)^2/(S\sum_{s=1}^{S}w_s^2)$,
measures the concentration of the importance weights among the posterior
draws. It does not account for their Markov dependence or unsampled modes.

{Across the Gaussian fits for each target, minimum Kish effective sample size fractions for shifted Poisson reweighting are $0.403$ and $0.482$ for $\beta=1/2$ and $\beta=1$. For power-law width they are $0.306$ and $0.194$, respectively.}

Table~\ref{tab:width-robustness} reports normalized errors and standard errors across the 50 original data replications in each case. The Gaussian width comparisons reweight the same posterior samples. This comparison does not establish a uniform rate over all width laws satisfying \eqref{eq:width}.

\begin{table}[htbp]
\caption{Squared prediction error divided by $q_{n,\beta}=(\log n/n)^{2\beta/(2\beta+1)}$ under three width laws with Gaussian parameters. Each estimate uses 50 independent replications. Parentheses give Monte Carlo standard errors. Supplementary Tables~\ref{tab:supp-half} and~\ref{tab:supp-lipschitz} give pointwise 95\% normal Monte Carlo confidence intervals.}
\label{tab:width-robustness}
\centering\small
\begin{tabular}{@{}crrrr@{}}
\toprule
$\beta$ & $n$ & Geometric & Shifted Poisson & Power law\\
\midrule
$1/2$ & 64 & 0.333 (0.015) & 0.335 (0.015) & 0.331 (0.014)\\
 & 128 & 0.383 (0.018) & 0.383 (0.018) & 0.383 (0.018)\\
 & 256 & 0.430 (0.009) & 0.429 (0.009) & 0.434 (0.009)\\
 & 512 & 0.460 (0.016) & 0.463 (0.015) & 0.477 (0.016)\\
 & 1024 & 0.278 (0.025) & 0.293 (0.025) & 0.287 (0.027)\\
 & 2048 & 0.164 (0.006) & 0.172 (0.006) & 0.161 (0.006)\\
 & 4096 & 0.132 (0.005) & 0.145 (0.006) & 0.128 (0.005)\\
\midrule
$1$ & 64 & 0.759 (0.025) & 0.762 (0.025) & 0.757 (0.024)\\
 & 128 & 0.941 (0.024) & 0.944 (0.022) & 0.953 (0.023)\\
 & 256 & 1.173 (0.034) & 1.189 (0.029) & 1.222 (0.029)\\
 & 512 & 0.769 (0.086) & 0.812 (0.086) & 0.809 (0.091)\\
 & 1024 & 0.317 (0.020) & 0.330 (0.024) & 0.321 (0.025)\\
 & 2048 & 0.277 (0.014) & 0.290 (0.014) & 0.273 (0.014)\\
 & 4096 & 0.255 (0.012) & 0.272 (0.012) & 0.249 (0.011)\\
\bottomrule
\end{tabular}
\end{table}

\subsection{Sensitivity to the parameter law}
\label{app:simulation-weights}
Uniform and Student $t_{10}$ simulations use geometric width and the computational methods above. Figure~\ref{fig:adaptive-risk} and the table below report 50 original diagnostic-passing data replications per case.

\begin{table}[htbp]
\caption{Squared prediction error for uniform and Student $t_{10}$ parameters with geometric width. Each estimate uses 50 independent replications. Parentheses give Monte Carlo standard errors. Supplementary Tables~\ref{tab:supp-half} and~\ref{tab:supp-lipschitz} give pointwise 95\% normal Monte Carlo confidence intervals.}
\label{tab:direct-family-adaptive-risk}
\centering\small
\begin{tabular}{@{}crrrrr@{}}
\toprule
& & \multicolumn{2}{c}{Uniform parameters} & \multicolumn{2}{c}{Student $t_{10}$ parameters}\\
\cmidrule(lr){3-4}\cmidrule(l){5-6}
$\beta$ & $n$ & Error & Error$/q_{n,\beta}$ & Error & Error$/q_{n,\beta}$\\
\midrule
$1/2$ & 64 & 0.0860 (0.0038) & 0.337 (0.015) & 0.0847 (0.0037) & 0.332 (0.015)\\
 & 128 & 0.0750 (0.0035) & 0.385 (0.018) & 0.0745 (0.0034) & 0.383 (0.018)\\
 & 256 & 0.0639 (0.0014) & 0.434 (0.009) & 0.0638 (0.0013) & 0.433 (0.009)\\
 & 512 & 0.0548 (0.0013) & 0.497 (0.011) & 0.0564 (0.0006) & 0.511 (0.006)\\
 & 1024 & 0.0307 (0.0025) & 0.373 (0.030) & 0.0349 (0.0019) & 0.424 (0.024)\\
 & 2048 & 0.0114 (0.0006) & 0.188 (0.010) & 0.0121 (0.0005) & 0.199 (0.008)\\
 & 4096 & 0.0072 (0.0003) & 0.160 (0.007) & 0.0068 (0.0002) & 0.151 (0.005)\\
\midrule
$1$ & 64 & 0.1236 (0.0040) & 0.765 (0.025) & 0.1228 (0.0036) & 0.760 (0.023)\\
 & 128 & 0.1077 (0.0026) & 0.955 (0.023) & 0.1086 (0.0023) & 0.963 (0.021)\\
 & 256 & 0.0954 (0.0022) & 1.228 (0.028) & 0.0999 (0.0014) & 1.285 (0.018)\\
 & 512 & 0.0542 (0.0051) & 1.023 (0.096) & 0.0808 (0.0023) & 1.526 (0.044)\\
 & 1024 & 0.0148 (0.0020) & 0.415 (0.057) & 0.0202 (0.0017) & 0.564 (0.047)\\
 & 2048 & 0.0073 (0.0004) & 0.305 (0.015) & 0.0081 (0.0004) & 0.337 (0.016)\\
 & 4096 & 0.0047 (0.0002) & 0.292 (0.013) & 0.0042 (0.0002) & 0.264 (0.012)\\
\bottomrule
\end{tabular}
\end{table}

\clearpage
\subsection{Supplementary numerical results}
\label{app:simulation-supplementary}
Supplementary Tables~\ref{tab:supp-half} and~\ref{tab:supp-lipschitz}
report the unnormalized mean squared errors, Monte Carlo standard errors
and confidence intervals for all prior comparisons at every sample
size. They use the same 50 replications and interval formula as
Figure~\ref{fig:adaptive-risk}.

\begin{table}[H]
\caption{Supplementary squared-error estimates for the H\"older $1/2$ target. Each mean uses 50 replications; SE is the Monte Carlo standard error and CI is the pointwise 95\% normal Monte Carlo confidence interval. N/U/$t_{10}$ denote Gaussian/uniform/Student parameters; G/P/H denote geometric/shifted Poisson/power-law width.}
\label{tab:supp-half}
\centering\small
\begin{tabular}{@{}rlrrrr@{}}
\toprule
$n$ & Prior & Mean error & SE & \multicolumn{2}{c}{95\% CI}\\
\cmidrule(l){5-6}
 & & & & Lower & Upper\\
\midrule
64 & N/G & 0.085007 & 0.003779 & 0.077601 & 0.092414\\
 & N/P & 0.085522 & 0.003859 & 0.077959 & 0.093085\\
 & N/H & 0.084430 & 0.003688 & 0.077202 & 0.091658\\
 & U/G & 0.085986 & 0.003842 & 0.078457 & 0.093516\\
 & $t_{10}$/G & 0.084667 & 0.003708 & 0.077399 & 0.091936\\
\addlinespace[3pt]
128 & N/G & 0.074534 & 0.003511 & 0.067653 & 0.081415\\
 & N/P & 0.074636 & 0.003561 & 0.067657 & 0.081615\\
 & N/H & 0.074575 & 0.003446 & 0.067820 & 0.081330\\
 & U/G & 0.075039 & 0.003455 & 0.068267 & 0.081812\\
 & $t_{10}$/G & 0.074530 & 0.003424 & 0.067818 & 0.081241\\
\addlinespace[3pt]
256 & N/G & 0.063275 & 0.001353 & 0.060623 & 0.065926\\
 & N/P & 0.063092 & 0.001354 & 0.060439 & 0.065746\\
 & N/H & 0.063818 & 0.001343 & 0.061187 & 0.066450\\
 & U/G & 0.063906 & 0.001390 & 0.061181 & 0.066631\\
 & $t_{10}$/G & 0.063767 & 0.001337 & 0.061146 & 0.066388\\
\addlinespace[3pt]
512 & N/G & 0.050776 & 0.001794 & 0.047259 & 0.054293\\
 & N/P & 0.051131 & 0.001679 & 0.047841 & 0.054421\\
 & N/H & 0.052641 & 0.001764 & 0.049185 & 0.056098\\
 & U/G & 0.054809 & 0.001265 & 0.052329 & 0.057289\\
 & $t_{10}$/G & 0.056407 & 0.000638 & 0.055156 & 0.057658\\
\addlinespace[3pt]
1024 & N/G & 0.022838 & 0.002027 & 0.018865 & 0.026811\\
 & N/P & 0.024076 & 0.002065 & 0.020029 & 0.028123\\
 & N/H & 0.023650 & 0.002197 & 0.019344 & 0.027955\\
 & U/G & 0.030668 & 0.002498 & 0.025772 & 0.035564\\
 & $t_{10}$/G & 0.034906 & 0.001949 & 0.031085 & 0.038727\\
\addlinespace[3pt]
2048 & N/G & 0.010029 & 0.000367 & 0.009309 & 0.010749\\
 & N/P & 0.010487 & 0.000387 & 0.009728 & 0.011246\\
 & N/H & 0.009851 & 0.000366 & 0.009133 & 0.010569\\
 & U/G & 0.011447 & 0.000598 & 0.010275 & 0.012620\\
 & $t_{10}$/G & 0.012139 & 0.000518 & 0.011125 & 0.013154\\
\addlinespace[3pt]
4096 & N/G & 0.005968 & 0.000239 & 0.005499 & 0.006436\\
 & N/P & 0.006533 & 0.000262 & 0.006020 & 0.007046\\
 & N/H & 0.005782 & 0.000230 & 0.005331 & 0.006233\\
 & U/G & 0.007191 & 0.000333 & 0.006539 & 0.007843\\
 & $t_{10}$/G & 0.006812 & 0.000227 & 0.006367 & 0.007256\\
\bottomrule
\end{tabular}
\end{table}

\begin{table}[H]
\caption{Supplementary squared-error estimates for the Lipschitz target. Each mean uses 50 replications; SE is the Monte Carlo standard error and CI is the pointwise 95\% normal Monte Carlo confidence interval. N/U/$t_{10}$ denote Gaussian/uniform/Student parameters; G/P/H denote geometric/shifted Poisson/power-law width.}
\label{tab:supp-lipschitz}
\centering\small
\begin{tabular}{@{}rlrrrr@{}}
\toprule
$n$ & Prior & Mean error & SE & \multicolumn{2}{c}{95\% CI}\\
\cmidrule(l){5-6}
 & & & & Lower & Upper\\
\midrule
64 & N/G & 0.122619 & 0.004002 & 0.114774 & 0.130463\\
 & N/P & 0.123184 & 0.004071 & 0.115206 & 0.131162\\
 & N/H & 0.122391 & 0.003926 & 0.114696 & 0.130085\\
 & U/G & 0.123569 & 0.003989 & 0.115751 & 0.131387\\
 & $t_{10}$/G & 0.122785 & 0.003641 & 0.115649 & 0.129921\\
\addlinespace[3pt]
128 & N/G & 0.106235 & 0.002660 & 0.101022 & 0.111449\\
 & N/P & 0.106525 & 0.002508 & 0.101610 & 0.111440\\
 & N/H & 0.107517 & 0.002581 & 0.102460 & 0.112575\\
 & U/G & 0.107741 & 0.002634 & 0.102579 & 0.112903\\
 & $t_{10}$/G & 0.108627 & 0.002319 & 0.104082 & 0.113172\\
\addlinespace[3pt]
256 & N/G & 0.091114 & 0.002667 & 0.085887 & 0.096341\\
 & N/P & 0.092422 & 0.002279 & 0.087955 & 0.096888\\
 & N/H & 0.094926 & 0.002222 & 0.090572 & 0.099281\\
 & U/G & 0.095434 & 0.002179 & 0.091163 & 0.099706\\
 & $t_{10}$/G & 0.099880 & 0.001431 & 0.097076 & 0.102684\\
\addlinespace[3pt]
512 & N/G & 0.040727 & 0.004566 & 0.031777 & 0.049677\\
 & N/P & 0.042978 & 0.004576 & 0.034008 & 0.051947\\
 & N/H & 0.042831 & 0.004841 & 0.033343 & 0.052320\\
 & U/G & 0.054153 & 0.005064 & 0.044228 & 0.064078\\
 & $t_{10}$/G & 0.080794 & 0.002322 & 0.076244 & 0.085344\\
\addlinespace[3pt]
1024 & N/G & 0.011347 & 0.000729 & 0.009917 & 0.012776\\
 & N/P & 0.011812 & 0.000862 & 0.010122 & 0.013501\\
 & N/H & 0.011479 & 0.000902 & 0.009711 & 0.013248\\
 & U/G & 0.014834 & 0.002039 & 0.010837 & 0.018830\\
 & $t_{10}$/G & 0.020169 & 0.001685 & 0.016867 & 0.023471\\
\addlinespace[3pt]
2048 & N/G & 0.006665 & 0.000336 & 0.006007 & 0.007323\\
 & N/P & 0.006968 & 0.000333 & 0.006316 & 0.007621\\
 & N/H & 0.006564 & 0.000336 & 0.005905 & 0.007222\\
 & U/G & 0.007315 & 0.000364 & 0.006601 & 0.008029\\
 & $t_{10}$/G & 0.008088 & 0.000381 & 0.007341 & 0.008835\\
\addlinespace[3pt]
4096 & N/G & 0.004082 & 0.000185 & 0.003720 & 0.004445\\
 & N/P & 0.004365 & 0.000188 & 0.003996 & 0.004734\\
 & N/H & 0.003999 & 0.000184 & 0.003639 & 0.004360\\
 & U/G & 0.004688 & 0.000202 & 0.004291 & 0.005085\\
 & $t_{10}$/G & 0.004230 & 0.000192 & 0.003853 & 0.004607\\
\bottomrule
\end{tabular}
\end{table}

\clearpage

\section{Parameter-law probability calculations}
\label{app:weight-boxes}

For an admissible centre $\theta^*$ in
\eqref{eq:weight-centres} and $0<\delta\leq1/2$, write
\[
\mathcal B_k(\theta^*,\delta)=
\{\theta\in\mathbb R^{(d+2)k}:
\norm{\theta-\theta^*}_\infty\leq\delta\}.
\]
We prove the bound \eqref{eq:weight-mass} for every such centre,
radius, and width, with constants independent of them.

\subsection{Uniform parameters}

The output intervals lie inside $[-k^\omega,k^\omega]$ because
$|a_m^*|+\delta\leq k^\omega/2+1/2\leq k^\omega$.
The hidden-layer coefficient and bias centres are bounded by one,
so their intervals lie inside $[-2,2]$. Independence gives
\begin{equation}\label{eq:uniform-exact}
\Pi_k^\theta\{\mathcal B_k(\theta^*,\delta)\}
=\left(\frac{\delta}{k^\omega}\right)^k
\left(\frac{\delta}{2}\right)^{(d+1)k}.
\end{equation}
Consequently,
\begin{equation}\label{eq:uniformcost}
-\log\Pi_k^\theta\{\mathcal B_k(\theta^*,\delta)\}
=k\omega\log k+(d+2)k\log(1/\delta)+(d+1)k\log2
\lesssim k\log(ek/\delta).
\end{equation}
This proves the uniform case of Lemma~\ref{lem:families}.
At a fixed width the parameter support is compact, and its continuous
image in $\Ccal(\Xcal)$ is compact. The full mixture allows
arbitrarily large widths and does not have compact function support.

\subsection{Gaussian parameters}

For $Z\sim\mathrm N(0,\tau^2)$, $\tau\geq1$, and $z\in\mathbb R$, integrating the
density on $[z-\delta,z+\delta]$ gives
\begin{equation}\label{eq:gauss-local}
\PP\{|Z-z|\leq\delta\}
\geq\frac{2\delta}{\sqrt{2\pi}\tau}
\exp\left\{-\frac{(|z|+\delta)^2}{2\tau^2}\right\}.
\end{equation}
Since $(|z|+\delta)^2\leq2z^2+2\delta^2$, the negative logarithm
is at most $C+\log(\tau/\delta)+z^2/\tau^2$.
The bounded hidden-layer coefficient and bias centres add only
$O(k)$ beyond the precision terms. With output standard deviation
$\tau_k=k^{3/(2d)}$, independence therefore yields
\begin{equation}\label{eq:gaussian-box-full}
\begin{split}
-\log\Pi_k^\theta\{\mathcal B_k(\theta^*,\delta)\}
\leq{}&Ck+(d+2)k\log(1/\delta)\\
&+k\log\tau_k+
\tau_k^{-2}\sum_{m=1}^k|a_m^*|^2.
\end{split}
\end{equation}
Admissibility controls the final term:
\begin{equation}\label{eq:gaussiancost}
\tau_k^{-2}\sum_{m=1}^k|a_m^*|^2
\leq k^{-3/d}\left(\sum_{m=1}^k|a_m^*|\right)^2
\leq\tfrac14 k^{-3/d}k^{2\omega}=\tfrac14 k.
\end{equation}
Also $k\log\tau_k=(3/(2d))k\log k$.
Thus \eqref{eq:gaussian-box-full} is
$O(k\log(ek/\delta))$, proving the Gaussian case of
Lemma~\ref{lem:families}.
For unit-variance Gaussian output parameters, an admissible centre
with $a_1^*=k^\omega/2$ and all other coordinates zero has a
quadratic log-probability cost of order $k^{2\omega}$.
Since $2\omega=1+3/d>1$, it fails
\eqref{eq:weight-mass} at any fixed positive radius. This failure
of the sufficient condition does not determine that prior's
posterior rate.

\subsection{Student \texorpdfstring{$t$}{t} and Cauchy parameters}

For every fixed $\nu>0$, the density of $t_\nu$ is
\[
q_\nu(z)=\frac{\Gamma((\nu+1)/2)}
{\sqrt{\nu\pi}\,\Gamma(\nu/2)}
\left(1+\frac{z^2}{\nu}\right)^{-(\nu+1)/2},
\qquad z\in\mathbb R.
\]
It satisfies $q_\nu(z)\geq c_\nu(1+|z|)^{-(\nu+1)}$.
If $Z$ has a probability density $q$ on $\mathbb R$ satisfying
$q(z)\geq c(1+|z|)^{-r}$ for all $z\in\mathbb R$, then
integration over $[z-\delta,z+\delta]$ gives
\begin{equation}\label{eq:t-local}
\PP\{|Z-z|\leq\delta\}
\geq2c\delta\,2^{-r}(1+|z|)^{-r}.
\end{equation}
For Student parameters, use $r=\nu+1$.
All hidden-layer and bias centres are bounded by one, and
admissibility gives $|a_m^*|\leq k^\omega/2$. Summing the
negative logarithms of the independent interval probabilities gives
\begin{equation}\label{eq:t-box-full}
\begin{split}
-\log\Pi_k^\theta\{\mathcal B_k(\theta^*,\delta)\}
\leq{}&Ck+(d+2)k\log(1/\delta)\\
&+(\nu+1)k\log(1+k^\omega/2)
\lesssim k\log(ek/\delta).
\end{split}
\end{equation}
This proves the Student case and completes the proof of
Lemma~\ref{lem:families}. The Cauchy law is $t_1$, and no moment
of the parameter law is used.

\subsection{Polynomial density lower bounds}
\label{app:polynomial-density}

\begin{proof}[Proof of Proposition~\ref{cor:polytail}]
Apply \eqref{eq:t-local} to each scalar parameter. As in the
Student calculation, all hidden-layer and bias centres are bounded
by one and $|a_m^*|\leq k^\omega/2$.
Independence bounds the negative log probability of
$\mathcal B_k(\theta^*,\delta)$ by
\[
Ck+(d+2)k\log(1/\delta)+rk\log(1+k^\omega/2)
\lesssim k\log(ek/\delta).
\]
The constant depends only on $d,c,r$. This proves
\eqref{eq:weight-mass}; paired with \eqref{eq:width}, the prior is
therefore covered by Proposition~\ref{prop:smallball} and the
contraction and risk theorems. No moment assumption is needed.
\end{proof}

\section{Fractional-posterior prior-probability argument}
\label{app:fractional-proof}

\begingroup
\begin{proof}[Proof of Lemma~\ref{lem:frac}]
We first give a lower bound for the posterior denominator and then
an upper bound for the expected numerator outside the stated
R\'enyi neighbourhood.
Let $\mathcal B_n=\{f\in\Ccal(\Xcal):
\norm{f-f_0}_\infty\leq\varepsilon_n\}$ and
$\Pi_{\mathcal B_n}=\Pi(\,\cdot\mid\mathcal B_n)$.
The likelihood ratio under $f$ relative to $f_0$ is
\[
\log R_n(f)=-\frac12\sum_{i=1}^n
\{(Y_i-f(X_i))^2-(Y_i-f_0(X_i))^2\}.
\]
The powered denominator is
$Z_n=\int_{\Ccal(\Xcal)}R_n(f)^\eta\Pi(\mathrm d f)$.
Restricting this integral to $\mathcal B_n$ and applying Jensen's
inequality gives
\begin{equation}\label{eq:den-main}
Z_n\geq\Pi(\mathcal B_n)e^{\eta\Lambda_n},
\qquad
\Lambda_n=\int_{\mathcal B_n}\log R_n(f)\,
\Pi_{\mathcal B_n}(\mathrm d f).
\end{equation}

Write $h_i(f)=f(X_i)-f_0(X_i)$.
Since $Y_i=f_0(X_i)+\xi_i$, expansion of the squares gives
\[
\Lambda_n
=\sum_{i=1}^n\xi_i\int_{\mathcal B_n}h_i(f)\,
\Pi_{\mathcal B_n}(\mathrm d f)
-\frac12\sum_{i=1}^n\int_{\mathcal B_n}h_i(f)^2\,
\Pi_{\mathcal B_n}(\mathrm d f).
\]
Conditional on the covariates, the integrals are deterministic
and the noises are independent standard normals. Thus
$\Lambda_n$ is Gaussian, with conditional mean and variance
\begin{align*}
\mu_n&=-\frac12\sum_{i=1}^n\int_{\mathcal B_n}h_i(f)^2\,
\Pi_{\mathcal B_n}(\mathrm d f)
\geq-\frac12n\varepsilon_n^2,\\
\sigma_n^2&=\sum_{i=1}^n
\left\{\int_{\mathcal B_n}h_i(f)\,
\Pi_{\mathcal B_n}(\mathrm d f)\right\}^2
\leq n\varepsilon_n^2.
\end{align*}
The bounds use $|h_i(f)|\leq\varepsilon_n$ on $\mathcal B_n$.
A Gaussian tail bound gives
$\PP_{f_0}\{\Lambda_n<-n\varepsilon_n^2\}
\leq e^{-n\varepsilon_n^2/8}$; the zero-variance case satisfies
the same bound. The assumed prior probability and
\eqref{eq:den-main} imply
\begin{equation}\label{eq:den-tail-main}
\PP_{f_0}\{Z_n<e^{-(c_0+\eta)n\varepsilon_n^2}\}
\leq e^{-n\varepsilon_n^2/8}.
\end{equation}

Set $\mathcal A_n=\{f\in\Ccal(\Xcal):
D_\eta(f,f_0)>J\varepsilon_n^2\}$.
Integrating the Gaussian noises conditional on the covariates gives
\[
\E_{f_0}\{R_n(f)^\eta\mid X_1,\ldots,X_n\}
=\exp\left\{-\frac{\eta(1-\eta)}2
\sum_{i=1}^n(f(X_i)-f_0(X_i))^2\right\}.
\]
Integrating the independent covariates and using \eqref{eq:renyi}
therefore yields
\[
\E_{f_0}R_n(f)^\eta
=e^{-(1-\eta)nD_\eta(f,f_0)}.
\]
By Tonelli's theorem,
\[
\E_{f_0}\int_{\mathcal A_n}R_n(f)^\eta\Pi(\mathrm d f)
\leq e^{-(1-\eta)Jn\varepsilon_n^2}.
\]

On the event $Z_n\geq e^{-(c_0+\eta)n\varepsilon_n^2}$,
the posterior probability of $\mathcal A_n$ is at most its numerator
multiplied by $e^{(c_0+\eta)n\varepsilon_n^2}$.
On the complementary event it is at most one.
Taking expectations and applying \eqref{eq:den-tail-main} gives
\[
\E_{f_0}\Pi_n^\eta(\mathcal A_n\mid\Dcal_n)
\leq e^{-n\varepsilon_n^2/8}
+e^{-\{(1-\eta)J-c_0-\eta\}n\varepsilon_n^2}.
\]
All estimates are uniform over the assumed class of targets,
proving \eqref{eq:frac-general}.
\end{proof}
\endgroup

\end{document}